\documentclass[journal,twoside]{IEEEtran}

\ifCLASSINFOpdf
\else
\fi

\usepackage{graphicx}
\usepackage{amsmath}
\usepackage[latin1]{inputenc}
\usepackage{tikz}
\usepackage{enumitem}
\usepackage{ragged2e}
\usepackage[all]{xy}
\usepackage{ amssymb }
\usepackage{amscd}
\usepackage{ dsfont }
\usepackage{epstopdf}
\usepackage{textcomp}
\usepackage{color}
\usepackage{array}
\usepackage{ amssymb }
\usepackage{ tipa }
\usepackage{ upgreek }
\usepackage{graphicx}
\usepackage{graphicx,latexsym} 
\usepackage{amssymb,amsmath}
\usepackage{longtable,booktabs,setspace} 
\DeclareMathOperator*{\minimise}{minimise}
\DeclareMathOperator*{\maximise}{maximise}
\newtheorem{remark}{Remark}

\newtheorem{theorem}{Theorem}[section]
\newtheorem{lemma}[theorem]{Lemma}

\newtheorem{proof}[theorem]{Proof}

\begin{document}
%

\title{Arbitrarily Tight Bounds on a Singularly Perturbed Linear-Quadratic Optimal Control Problem}
%
%
%

\author{Sei~Howe, 
Panos~Parpas\thanks{S. Howe and P. Parpas are with the Department of Computer Science, Imperial College London, U.K e-mail: sei.howe11@imperial.ac.uk, panos.parpas@imperial.ac.uk. }
\thanks{Thanks to support from EPSRC grants EP/M028240, EP/K040723 and an FP7 Marie Curie Career Integration Grant (PCIG11-GA-2012-321698 SOC-MP-ES).}}%
\maketitle

\begin{abstract}
We calculate arbitrarily tight upper and lower bounds on an unconstrained control, linear-quadratic, singularly perturbed optimal control problem whose exact solution is computationally intractable.   It is well known that  for the aforementioned problem, an approximate solution $\bar{V}^N(\epsilon)$ can be constructed such that it is  asymptotically equivalent in $\epsilon$ to the solution $V(\epsilon)$ of the singularly perturbed problem in the sense that $|V(\epsilon)-\bar{V}^N(\epsilon)| =O(\epsilon^{N+1})$  for any integer $N\geq0$ as $\epsilon \rightarrow 0$.   For this approximation to be considered useful, the parameter $\epsilon$ is typically restricted to be in some sufficiently small set; however,   for values of $\epsilon$ outside this set, a poor approximation can result. We improve on this approximation by incorporating a duality theory into the singularly perturbed optimal control problem and derive an upper bound $\chi^N_u(\epsilon)$ and a lower bound  $\chi^N_l(\epsilon)$ of $V(\epsilon)$ that hold for arbitrary $\epsilon$ and, furthermore, satisfy the  inequality  $|\chi^N_u(\epsilon)-\chi^N_l(\epsilon)|=O(\epsilon^{N+1})$ for any integer $N \geq 0$ as $\epsilon \rightarrow 0$.   
\end{abstract}



\IEEEpeerreviewmaketitle

\section{Introduction}
\label{Introduction} 
\IEEEPARstart{S}{ingularly} 
perturbed optimal control (SPOC) problems are characterised by the presence of a small  parameter $\epsilon$ multiplying the highest derivative of some of the dynamics of the system. This parameter, known as a singular perturbation parameter, results in a system where some variables change at  a much faster rate than others; thus indicating that the system possesses a two time-scale separation.   This time-scale separation frequently leads to computational issues as it introduces stiffness into the optimal control problem.  Furthermore, these computational issues can often be compounded by the curse of dimensionality that arises in large-scale systems.  In order to counteract these two problems,  various authors have devised computationally feasible methods of obtaining an approximation to the solution.  In particular, for the unconstrained control, linear-quadratic, SPOC problem, \cite{OMalley3} - \cite{OMalleyKung} obtained a computational feasible approximation  $V^N(\epsilon)$ to the solution $V(\epsilon)$  satisfying the following bound \begin{equation} \label{old old result}
|V(\epsilon)-V^N(\epsilon)| =O(\epsilon^{N+1}), \hspace{4mm} \mbox{ as } \epsilon \rightarrow 0.
\end{equation}
Although the approximation in \eqref{old old result} holds for any $N\geq 0$, the singular perturbation parameter is restricted to be in some sufficiently small set for the purpose of obtaining a useful approximation.  In this paper, we improve on the result in \eqref{old old result} by determining an upper bound $\chi^N_u(\epsilon)$ and a lower bound $\chi^N_l(\epsilon)$ on $V(\epsilon)$ that hold for arbitrary values of $\epsilon$  and, furthermore, satisfy the bound
\begin{eqnarray}\label{old result}
|\chi^N_l(\epsilon)-\chi_u^N(\epsilon)|=O( \epsilon^{N+1}), \hspace{4mm} \mbox{ as } \epsilon \rightarrow 0.\end{eqnarray}
In particular, our result allows the practitioner to determine a region in which the solution is contained for values of $\epsilon$ that do not necessarily fall within some sufficiently small set but where the solution is still impractical to compute numerically.   

There exists  a vast array of physical problems that fit the linear-quadratic, unconstrained control, SPOC structure  but with an  $\epsilon$ parameter that is considerably larger than what may be considered to be sufficiently small.   A simply supported beam example is considered in \cite{Hsieh} and power systems are considered in \cite{Arkun} both with $\epsilon=0.1$. Flight control systems are considered in \cite{Brumbaugh}, \cite{Nguyen} and \cite{Shin} with $\epsilon$ set as $0.2$, $0.336$ and $0.0424$ respectively.    More recently, many consensus network and graph aggregation problems have been considered in \cite{Boker}, \cite{Boker2}, \cite{Chow}, \cite{Ishii} and \cite{Rejab} for various values of $\epsilon$. For such problems, an asymptotic result of the form \eqref{old old result} may be inadequate as an approximation. 

As our methodology provides definitive bounds on the solution for all values of $\epsilon$, it both increases the amount of information available when considering the implementation of an approximate solution and  produces a criterion for determining how good of an approximation  $V^N$ yields. While the authors' previous work in \cite{SeiPanos} established upper and lower bounds on the solution to  a control constrained, linear-quadratic SPOC problem of the form
\begin{align*}
|\chi_l(\epsilon)-\chi_u(\epsilon)|=O( \epsilon), \hspace{4mm} \mbox{ as } \epsilon \rightarrow 0,
\end{align*}
to the extent of the authors knowledge, the derivation of arbitrarily tight upper and lower bounds satisfying \eqref{old result}  on the solution to a unconstrained control, linear-quadratic SPOC problem  and a criteria to evaluate an asymptotically optimal approximation have never been previously considered.  

As the optimal control problem that we consider is a minimisation problem,  an upper bound on the solution is easily found by evaluating the problem with any feasible control.   By evaluating the linear dynamics of the problem with an asymptotic expansion of the optimal control obtained using a reduced dimension problem,  one obtains an arbitrarily tight upper bound.  An arbitrarily tight lower bound, however, has been more difficult to obtain due to a lack of a duality framework in which to formulate the SPOC problem and, moreover, a lack of a strong duality property that ensures that any arbitrarily tight lower bound to the dual problem will also be an arbitrarily tight lower bound to the primal problem.  

In this paper, we apply the duality construction in \cite{Alt} and \cite{Burachik} to the case of SPOC problems and derive a dual problem with the strong duality property.    From duality theory, it follows that the dual problem evaluated with any feasible control will provide a lower bound on the  solution of the SPOC problem.     In order to obtain an arbitrarily tight lower bound, we use the strong duality result and the asymptotic expansion  of the optimal control of the primal problem  to construct an asymptotic expansion of the optimal  control of the dual problem.  By evaluating the dual problem with the constructed control, we obtain an arbitrarily tight lower bound to the optimal control problem.   

We illustrate our results with three examples: one relating to aircraft control  and two over a clustered consensus network.   In the first two examples, we present the solution to the optimal control problem, the approximate solution $V^N$ obtained in \cite{OMalley3}-\cite{OMalleyKung}, and our upper and lower bounds.  In the third problem we consider, the solver was not able to obtain the solution to the control problem; however, we were able to obtain both a reduced solution and bounds. We show that in all cases, the upper and lower bounds can provide a better approximation to the solution than $V^N$. In particular, for the aircraft example, we show that our upper and lower bounds provide a better approximation for $\epsilon <0.035$ and for the consensus network examples, we show that for $\epsilon =0.25$ and $\epsilon=0.0125$, where $\epsilon$ has been determined by the network topology,  the difference between the upper and lower bounds is $0.173$  and $0.0284$ respectively  and, in both cases, $V^N$ lies outside of these bounds.  From our results, it is clear that one now has a method for determining whether the approximation $V^N$ is  adequate for the purposes of the problem and, furthermore, if it is not adequate, a method of obtaining a better approximation to the solution.  

 In the following section we outline the singularly perturbed optimal control problem under consideration and present its dual formulation.  In  Section \ref{main results} we present our main theorems relating to the arbitrarily tight upper and lower bounds satisfying \eqref{old result}. Section \ref{construction} provides a brief description of the construction of the asymptotic expansion of the optimal control of our problem. The proofs of the main theorems are presented in Section \ref{main proofs} and the construction of the dual problem is presented in Section \ref{dual construction}. Finally, in Section \ref{numerical}, we provide our three examples.

 \section{Formulation of the primal and dual problems}
We consider the following  problem 
\begin{align*} \tag{\textbf{P}} \begin{cases} \underset{\hat{u}}\minimise&\displaystyle J_\bold{P}(\hat{z}_1,\hat{z}_2,\hat{u},\epsilon) ,
\\
\mbox{subject to} &
 \displaystyle \hspace{0.5mm}\frac{d\hat{z}_1}{dt}\hspace{1mm}=A_{11}\hat{z}_1+A_{12}\hat{z_2}+b_1\hat{u}, \vspace{2mm}\\
& \displaystyle\epsilon \frac{d\hat{z}_2}{dt}=A_{21}\hat{z}_1+A_{22}\hat{z}_2+b_2\hat{u},\\
& \hat{z}_1(0,\epsilon) = z_{1,0}(\epsilon),  \hspace{4mm} \hat{z}_2(0,\epsilon) = z_{2,0}(\epsilon),   \end{cases}
\end{align*}
for $ \epsilon \in (0,\epsilon^*]$, where the functional $J_\bold{P}$ is defined as
\begin{equation} \label{jp} 
J_\bold{P}= \frac{1}{2}\int_{0}^{1}  \hspace{-1.5mm}\hat{z}^TQ\hat{z}+\hat{u}^TR\hat{u} \hspace{1mm} \mathrm{d}t  + \frac{1}{2} \hat{z}(1,\epsilon)^T \pi(\epsilon) \hat{z}(1,\epsilon).
\end{equation}
and $\hat{z}^T = \begin{bmatrix} \hat{z}_1^T, \hat{z}_2^T\end{bmatrix}$.
We let $V_\bold{P}(\epsilon)$ denote the value of $J_\bold{P}$ evaluated at the optimal control, denoted by $u$, for fixed $\epsilon$.    The matrices $Q$, $R$, $A_{ij}$, $b_i$ for $i,j=1,2$ may depend on both $t$ and $\epsilon$. 
 For all  $ \epsilon \in (0,\epsilon^*]$,  $\hat{z}_1 \in   W ^{1,2} ([0,1];\mathds{R}^m )$, $\hat{z}_2~\in~ W ^{1,2} ([0,1];\mathds{R}^n)$,  and $\hat{u} \in W^{1,2}([0,1];\mathds{R}^k)$ as functions of $t$, where the space $W^{1,2}$ denotes the Sobelov space of absolutely continuous functions.  Note that for $\epsilon=0$, the dimension of the problem \textbf{P} drops from $m+n$ to $m$ and the  boundary condition for $\hat{z}_{2}$ may no longer be satisfied.  
 
The following standard assumptions are imposed on $\bold{P}$:
 \begin{itemize}
 \item[(a)] The matrix $A_{22}$ is negative definite for all $t\in[0,1]$, $\epsilon \in [0,\epsilon^*]$, 
\item[(b)] For any fixed $\epsilon \in [0, \epsilon^*]$, the matrices  $Q$, $R$, $A_{ij}$, and $b_i$ for $i,j=1,2$,
 are smooth for $t\in [0,1]$, 
 \item[(c)] $Q$, $R$, $\pi$, $A_{ij}$, and $b_i$, for $i,j=1,2$, all have an asymptotic expansion in $\epsilon$ which is valid over their respective domains.  
 \item[(d)] $R$ is positive definite  for all $t\in[0,1]$, $\epsilon \in [0,\epsilon^*]$,
  \item[(e)]  For $\epsilon\in [0,\epsilon^*]$, $\pi$ has the following block-diagonal structure
\begin{equation*}\label{pi}
\pi(\epsilon)=\begin{bmatrix}\pi_{11}(\epsilon) & \epsilon \pi_{12}(\epsilon) \\ \epsilon \pi_{12}(\epsilon)^T& \epsilon \pi_{22}(\epsilon) \end{bmatrix},
\end{equation*}
 where $\pi_{11}\in \mathds{R}^{m\times m},\hspace{1mm} \pi_{22}\in \mathds{R}^{n\times n}$,
  \item[(f)] $Q$ and $\pi$ are positive semi-definite for all $t\in[0,1]$, $\epsilon \in [0,\epsilon^*]$,
   \item[(g)] The eigenvalues of 
\begin{equation*}
G(t)=\begin{bmatrix} A_{22}^0(t) & -b_2^0(t)R^0(t)^{-1}b_2^0(t)^T\\ -Q_{22}^0(t) & -A_{22}^0(t)^T \end{bmatrix},
\end{equation*}
have non-zero real parts on $[0,1]$, where the superscript $0$ denotes the first term in the asymptotic expansion in $\epsilon$ of the appropriate matrix,
\item[(h)] There exists a non-singular matrix 
\begin{align} \label{tmatrix}
T(t)= \begin{bmatrix} T_{11}(t) & T_{12}(t) \\ T_{21}(t) & T_{22}(t)\end{bmatrix},
\end{align}
 such that 
\begin{align}\label{tlambda}
T(t)^{-1}G(t)T(t)= \begin{bmatrix} -\Lambda(t) &0 \\0&\Lambda(t)\end{bmatrix},
\end{align}
with all eigenvalues of $\Lambda(t)$ having positive real parts on $[0,1]$ and such that the matrices 
\begin{align*} \begin{split}
T_{11}(0), \hspace{4mm}T_{22}(1) - \pi_{22}^0T_{12}(1),
\end{split}
\end{align*}
are non-singular.

 \end{itemize}

 The assumptions $(a)-(h)$ are consistent with the assumptions in \cite{OMalleyKung}.     We further impose the condition that $Q$ and $\pi$ are positive definite and symmetric  which is necessary for the construction of the dual problem and the proof of strong duality.
The feasible set for $\textbf{P}$ can be written as
\begin{align}
\begin{split} \label{feasible set}
\Sigma&= \bigg{\{}(\hat{z},\hat{u}) : \hat{z} \in W^{1,2}([0,1];\mathds{R}^{m+n}) ,\hat{z}(0,\epsilon)=z_0(\epsilon),   \\& I^{\epsilon}\frac{d\hat{z}}{dt} =A\hat{z}+ b\hat{u},\hspace{2mm} t\in [0,1] , \hspace{2mm} \epsilon \in (0,\epsilon^*]\bigg{\}},
\end{split}
\end{align}
where $z_0^T=[z_{1,0}^T,z_{2,0}^T]$ and
\begin{align*}
A&=\begin{bmatrix} A_{11} &A_{12} \\ A_{21}& A_{22} \end{bmatrix}, \hspace{1mm} A_{11} \in \mathds{R}^{m\times m}, \hspace{1mm} A_{22} \in \mathds{R}^{n\times n}, \\
b &= \begin{bmatrix} b_1\\  b_2 \end{bmatrix}, \hspace{13mm} b_1\hspace{2mm}\in\mathds{R}^{m\times k},\hspace{2mm} b_2\hspace{2mm} \in \mathds{R}^{n\times k},
\end{align*}
\begin{align*}
I^\epsilon&=\begin{bmatrix} I_{m}&0\\ 0 &\epsilon I_n\end{bmatrix}.
\end{align*}
and $I_j$ is the $j\times j$ identity matrix for $j=m,n$. 

\begin{remark} \label{unique solution primal}
We assume the set $\Sigma$ is nonempty. Since the solution set $\Sigma$ is closed and convex and $J_\bold{P}$ is strictly convex, continuous and coercive over $\Sigma$, there exists a unique solution to  the minimisation problem $\bold{P}$ \cite{EkelandTemam}.
\end{remark}
 The construction of the dual problem is based on that of the unperturbed case \cite{Alt} and \cite{Burachik}.
The dual problem can be formulated as

\begin{align*} \tag{\textbf{D}} \begin{cases} \underset{\hat{\rho}_1,\hat{\rho}_2,\hat{\gamma}_1,\hat{\gamma}_2}\maximise&\displaystyle J_\bold{D}(\hat{\rho}_1,\hat{\rho}_2,\hat{\gamma}_1,\hat{\gamma}_2,\epsilon),
\\
\mbox{subject to} & \displaystyle \hspace{1mm}  \frac{d\hat{\gamma}_1}{dt}=-A_{11}^T\hat{\gamma}_1 -A_{21}^T \hat{\gamma}_2+\hat{\rho}_1,\vspace{2mm}\\
& \displaystyle \epsilon \frac{d\hat{\gamma}_2}{dt}=-A_{12}^T\hat{\gamma}_1-A_{22}^T\hat{\gamma}_2+\hat{\rho}_2, \end{cases}
\end{align*}
where the functional $J_{\textbf{D}}$ is given by 
\begin{align}\label{jd} \begin{split}
J_\bold{D}&= \frac{1}{2} \int_0^1\hspace{-2mm} -\hat{\rho}^TQ^{-1}\hat{\rho}-\hat{\gamma}^TbR^{-1}b^T\hat{\gamma} \mathrm{d}t - \hat{\gamma}(0,\epsilon)^T I^\epsilon z_0(\epsilon)\\
&-\frac{1}{2} \hat{\gamma}(1,\epsilon)^TI^\epsilon\pi(\epsilon)^{-1}I^\epsilon \hat{\gamma}(1,\epsilon), \end{split}
\end{align}
and $\hat{\rho}^T=\begin{bmatrix} \hat{\rho}_1^T, \hat{\rho}_2^T \end{bmatrix} $, $\hat{\gamma}^T=\begin{bmatrix}\hat{\gamma}_1^T,\hat{\gamma}_2^T \end{bmatrix}$.   We let $V_\bold{D}(\epsilon)$ denote the value of $J_\bold{D}$ evaluated at the optimal control, denoted by $\rho$, for fixed $\epsilon$. For all  $ \epsilon ~\in~ (0,\epsilon^*]$,  $\hat{\gamma}_1, \hat{\rho}_1\in W^{1,2}([0,1];\mathds{R}^{m})$, and  $\hat{\gamma}_2, \hat{\rho}_2~\in~ W^{1,2}([0,1];\mathds{R}^{n})$ as functions of $t$. 
 The feasible set for $\textbf{D}$ can be written as
 \begin{align*}
 \Sigma_1\hspace{-1mm}=&\hspace{-1mm} \bigg{\{}\hspace{-0.8mm}(\hat{\gamma}, \hat{\rho} )\hspace{-0.5mm}: \hspace{-0.5mm}\hat{\gamma},\hat{\rho} \in W^{1,2}([0,1]; \mathds{R}^{m+n}),\\& I^{\epsilon} \frac{d\hat{\gamma}}{dt} \hspace{-1mm}=\hspace{-1mm}- A^T\hat{\gamma}\hspace{-0.5mm}+\hspace{-0.5mm}\hat{\rho}, t\in\hspace{-0.5mm}[0,1], \epsilon \in \hspace{-0.5mm}(0,\epsilon^*] \hspace{-0.5mm}\bigg{\}}.
 \end{align*}
\begin{remark} \label{unique solution dual}
 The objective functional $J_\bold{D}$ is not necessarily coercive over the feasible set $\Sigma_1$; hence, we  cannot immediately conclude that a unique solution exists.  However, as there exists a unique solution to $\bold{P}$ by Remark \ref{unique solution primal}, the strong duality result in Section \ref{main results} will lead to the existence of a unique solution for $\bold{D}$.
\end{remark}

\section{Main results} \label{main results} 
In this section, we detail three theorems which compose our main results.  Theorem \ref{expansionomalley}  states that the optimal control of $\bold{P}$ admits an asymptotic expansion as $\epsilon \rightarrow 0$, Theorem \ref{strong duality}  states that strong duality, which is necessary for obtaining an arbitrarily tight lower bound to $\bold{P}$, holds, and Theorem \ref{main theorem} constructs the arbitrarily tight upper and lower bounds on the solution.  

 \begin{theorem}
\label{expansionomalley}
The optimal control $u$ of $\bold{P}$ has an asymptotic expansion of the form
\begin{align} \label{un1}
u(t,\epsilon) & =u^N(t,\epsilon)+O(\epsilon^{N+1}), \hspace{4mm} \mbox{ as } \epsilon \rightarrow 0,
\end{align}
for any integer $N \geq 0$, uniformly on $[0,1]$.  
\end{theorem}
The proof and construction of the expansion are contained within \cite{OMalley} and \cite{OMalleyKung}.   We will briefly outline the method for obtaining the $u^N$ term for any integer $N\geq 0$ in Section \ref{construction}. 
 
 \begin{theorem}\label{strong duality} 
The solution $V_{\bold{P}}(\epsilon)$ of the primal problem $\bold{P}$ and solution $V_{\bold{D}}(\epsilon)$ of the dual problem $\bold{D}$ satisfy the following equality
 \begin{equation*}
 V_{\bold{P}}(\epsilon)=V_{\bold{D}}(\epsilon), \hspace{2mm} \mbox{ for  } \epsilon \in (0,\epsilon^*] .
 \end{equation*}
\end{theorem}
The strong duality property in Theorem \ref{strong duality}  is crucial for  obtaining the  upper and lower bounds in equation \eqref{old result} as it implies that an arbitrarily tight lower bound for $V_{\bold{D}}(\epsilon)$ will also be an arbitrarily tight lower bound for $V_\bold{P}(\epsilon)$. Furthermore, as the dual objective functional $J_\bold{D}$ is  not necessarily coercive over its feasible set, the strong duality result, along with the existence of a unique solution to $\bold{P}$,  will imply that a unique solution exists for $\bold{D}$. 
\begin{theorem} \label{main theorem}\begin{enumerate}[label=(\alph*)]
\item The control $u^N(t,\epsilon)$ given in \eqref{un1} provides an asymptotically optimal upper bound to the solution of $\bold{P}$ in the sense that
 \begin{align}\label{pvjepsilon}
 V_\bold{P}(\epsilon)&=J_\bold{P}(u^N,\hat{z}^N,\epsilon)+O(\epsilon^{N+1}), 
 \end{align}
 where
 \begin{align*}
 & O(\epsilon^{N+1}) <0 \hspace{2mm}\mbox {as}\hspace{2mm} \epsilon \rightarrow 0^+,
 \end{align*}
for any integer $N\geq0$. 

 In \eqref{pvjepsilon}, $V_{\bold{P}}(\epsilon)$ is the solution to $\bold{P}$, $J_\bold{P}$ is the objective functional in \eqref{jp}, and  $\hat{z}^N$ is the state satisfying the differential equations  and boundary conditions in $\bold{P}$ with control given by $u^N$.  
\item Consider the following control 
 \begin{align}
\label{bold lambda control}
\hat{\rho}^N(t,\epsilon) &= Q(t,\epsilon)\hat{z}^N(t,\epsilon),
 \end{align} 
 where $\hat{z}^N$ is the state satisfying the differential equations and boundary conditions in $\bold{P}$ with control given by $u^N$.
The control $\hat{\rho}^N$ provides an asymptotically optimal lower bound to the solution of $\bold{D}$ in the sense that 
 \begin{align} \label{o epsilon dual}
 V_\bold{D}(\epsilon)= J_\bold{D}(\hat{\rho}^N,\hat{\gamma}^N,\epsilon )+ O(\epsilon^{N+1}),
 \end{align}
 where
 \begin{align*}
 O(\epsilon^{N+1}) >0 \hspace{2mm}\mbox {as}\hspace{2mm} \epsilon \rightarrow 0^+,
 \end{align*}
   for any integer $N\geq0$.   In \eqref{o epsilon dual}, $V_{\bold{D}}(\epsilon)$ is the solution to $\bold{D}$, $J_\bold{D}$ is the objective functional in \eqref{jd}, and  $\hat{\gamma}^N$ satisfies the differential equations in \textbf{D} with control  given by $\hat{\rho}^N$ and boundary condition
 \begin{align}\label{gamma1}
 \hat{\gamma}^N(1,\epsilon) = -I^{\frac{1}{\epsilon}} \pi(\epsilon) Q(1,\epsilon)^{-1} \hat{\rho}^N(1,\epsilon).
 \end{align}
\item The following inequality and asymptotic result holds \begin{align*} 
&J_{\bold{D}} (\hat{\rho}^N, \hat{\gamma}^N,\epsilon) \leq V_{\bold{D}}(\epsilon) = V_{\bold{P}} (\epsilon) \leq J_{\bold{P}}(u_N,\hat{z}^N,\epsilon), \\&\bigg{|}J_{\bold{P}} (u^N,\hat{z}^N,\epsilon)-J_{\bold{D}}(\hat{\rho}^N, \hat{\gamma}^N, \epsilon )\bigg{|} =O(\epsilon^{N+1}),
 \end{align*}
 as $\epsilon \rightarrow 0$ with $\hat{z}^N$ and $\hat{\gamma}^N$ as in Theorem 2, parts (a) and (b) respectively.    
 \end{enumerate}
 \end{theorem}
 %
 
 Parts (a) and (b) of Theorem \ref{main theorem} are proved in Section \ref{main proofs}.  Part (c) follows immediately from Theorem \ref{strong duality} and parts (a) and (b) in Theorem \ref{main theorem}, hence we will not explicitly go over the proof in this paper. 

\section{Construction of asymptotic expansion of the optimal control to $\bold{P}$} \label{construction}
In this section, we give a brief outline of the construction of an asymptotic expansion to the optimal control of $\bold{P}$ and $\bold{D}$.  
The full details may be found in \cite{OMalley}-\cite{OMalleyKung}.
We begin by deriving the necessary optimality conditions for $\bold{P}$ from the corresponding Hamiltonian function. Omitting dependence on $t$ and $\epsilon$ for simplicity, the  Hamiltonian function associated with a singularly perturbed problem of the form $\bold{P}$ (see \cite{KokotovicKhalilOreilly}, Chap. 6) is defined as 
\begin{equation*} \label{hamiltonian primal}
H^\bold{P}(\hat{z},\hat{u},\hat{\chi})=\frac{1}{2}(\hat{z}^T Q \hat{z}+ \hat{u}^TR\hat{u})+\hat{\chi}^T(A \hat{z}+ b\hat{u}) , 
\end{equation*}
where $\chi^T = [\chi_1^T,\chi_2^T]$ and $\chi_1 \in W^{1,2} ([0,1];\mathds{R}^m)$ and $\chi_2 \in W^{1,2} ([0,1];\mathds{R}^n)$ as functions of $t$ are the co-state variables associated with $z_1$ and $z_2$ respectively.  The scaling $\hat{\chi} \rightarrow I^{\frac{1}{\epsilon}}\hat{ \chi} $ recovers the standard form of the Hamiltonian.   Since there are no boundary conditions at the terminal time we have a normal Hamiltonian multiplier.
Let $u$, $z$, $\chi$ denote the optimal control, state and co-state respectively of $\bold{P}$.  These variables must satisfy the following necessary optimality conditions (see \cite{Clarke})
\begin{align}\begin{split}\label{p opt cond}
\frac{dz_1}{dt} & = A_{11}(t)z_1 + A_{12}(t) z_2 -S_{11} (t)\chi_{1} -S_{12}(t)\chi_{2} , \\
\frac{d\chi_1}{dt} & = -A_{11}^T (t)\chi_1 - A_{21}^T(t)\chi_2 - Q_{11}(t)z_1 -Q_{12} (t)z_2, \\
\epsilon \frac{dz_2}{dt} & = A_{21}(t) z_1 + A_{22}(t)z_2 -S_{12}^T(t) \chi_{1} -S_{22}(t) \chi_{2},  \\
\epsilon \frac{d\chi_2}{dt} & = -A_{12}^T (t)\chi_1 - A_{22}^T(t)\chi_2 - Q_{21}(t)z_1 -Q_{22}(t)z_2,\end{split}
\end{align}
where $S_{11}$, $S_{12}$ and $S_{22}$ are defined as
\begin{align*}
S_{11} & = b_1R^{-1}b_1^T ,\\
S_{12}&=b_1R^{-1}b_2^T, \\
S_{22}&= b_2R^{-1}b_2^T . 
\end{align*}
 The  boundary conditions that these variables must satisfy are given by
\begin{align}\begin{split} \label{p bound cond1}
z_1(0,\epsilon) = z_{1,0}, & \hspace{4mm} \chi_1(1,\epsilon) = \pi_{11}(\epsilon)z_1(1) + \epsilon \pi_{12}(\epsilon)z_{2}(1),\\
z_2(0,\epsilon)= z_{2,0}, & \hspace{4mm} \chi_{2}(1,\epsilon)= \pi_{21}(\epsilon)z_1(1) + \pi_{22}(\epsilon)z_{2}(1).
\end{split}
\end{align}
From the Pontryagin Minimum Principle \cite{Clarke}, it follows that the optimal control of $\bold{P}$ satisfies 
\begin{equation}\label{controlu}
u(t,\epsilon) = -R^{-1}(b_1^T\chi_1+b_2^T\chi_2),
\end{equation}
where $\chi_1$ and $\chi_2$ must satisfy the equations in \eqref{p opt cond}.    By Remark \ref{unique solution primal}, it follows that the necessary conditions are also sufficient; therefore any solution satisfying these conditions will be the unique solution to $\bold{P}$.  
 It follows from \eqref{controlu} and assumption (c) that in order to obtain an asymptotic expansion for $u$, we must obtain an asymptotic expansion for the co-state variables.   
 
 In the following theorem, we use the method of matched asymptotic expansions in order to derive an expansion for the optimal states and co-states on the outer layer as well as a boundary layer near the initial time and a boundary layer near the final time.
\begin{theorem}
Let us define the following time scales
\begin{equation*}
\tau = \frac{t}{\epsilon}, \hspace{5mm} \sigma=\frac{1-t}{\epsilon}.
\end{equation*}
The optimal states $z_1$ and $z_2$ and co-states $\chi_1$ and $\chi_2$ of the problem $\bold{P}$  have a unique asymptotic solution of the form
\begin{align} \begin{split} \label{expansionfull}
z_1(t,\epsilon)& = z_{1,o}(t,\epsilon) + \epsilon z_{1,i}(\tau,\epsilon) + \epsilon z_{1,f}(\sigma,\epsilon),\\
z_2(t,\epsilon)& = z_{2,o}(t,\epsilon) + z_{2,i}(\tau,\epsilon) +  z_{2,f}(\sigma,\epsilon),\\
\chi_1(t,\epsilon)& = \chi_{1,o}(t,\epsilon) + \epsilon \chi_{1,i}(\tau,\epsilon) + \epsilon \chi_{1,f}(\sigma,\epsilon),\\
\chi_2(t,\epsilon)& = \chi_{2,o}(t,\epsilon) + \chi_{2,i}(\tau,\epsilon) + \chi_{2,f}(\sigma,\epsilon).
\end{split}
\end{align}
The terms $z_{1,o}$, $z_{2,o}$, $\chi_{1,o}$ and $\chi_{2,o}$, known as the outer variables, satisfy the system \eqref{p opt cond} and have an asymptotic expansion in $\epsilon$. 
The terms $ z_{1,i}$, $z_{2,i}$ $ \chi_{1,i}$ and $ \chi_{2,i}$, known as the inner variables,  satisfy the system
\begin{align}
\begin{split} \label{ziqi}
\frac{dz_{1,i}}{d\tau} & = \epsilon A_{11}(\epsilon \tau ) z_{1,i} + A_{12}(\epsilon \tau) z_{2,i}- \epsilon S_{11}(\epsilon\tau)\chi_{1,i} \\&- S_{12}(\epsilon\tau) \chi_{2,i}\\
\frac{d\chi_{1,i}}{d\tau} & = -  \epsilon A_{11}^T(\epsilon \tau) \chi_{1,i} -  A_{21}^T(\epsilon \tau) \chi_{2,i}  - \epsilon Q_{11}(\epsilon \tau) z_{1,i} \\& -Q_{12}(\epsilon \tau) z_{2,i}   \\
\frac{dz_{2,i}}{d\tau} & = \epsilon A_{21}(\epsilon \tau)  z_{1,i} + A_{22}(\epsilon \tau) z_{2,i} - \epsilon S_{12}^T(\epsilon\tau)\chi_{1,i}\\& - S_{22}(\epsilon\tau) \chi_{2,i}\\
\frac{d\chi_{2,i}}{d\tau} & =  -  \epsilon A_{12}^T(\epsilon \tau) \chi_{1,i} -  A_{22}^T(\epsilon \tau) \chi_{2,i}  - \epsilon Q_{21}(\epsilon \tau ) z_{1,i}  \\&-Q_{22}(\epsilon \tau) z_{2,i}   ,
\end{split}
\end{align}
and have an asymptotic expansion in $\epsilon$. 
The terms  $ z_{1,f}$, $z_{2,f}$ $ \chi_{1,f}$ and $\chi_{2,f}$, known as the final variables,  satisfy the system
\begin{align} \label{zfqf}
\begin{split}
\frac{dz_{1,f}}{d\sigma} =&  -\epsilon A_{11}(1- \epsilon \sigma) z_{1,f} - A_{12}(1 - \epsilon\sigma)  z_{2,f} \\&+ \epsilon S_{11}(1 - \epsilon\sigma)\chi_{1,f}  + S_{12}(1 - \epsilon\sigma) \chi_{2,f}\\
\frac{d\chi_{1,f}}{d\sigma}  =&  \epsilon A_{11}^T(1-\epsilon \sigma) \chi_{1,f}  + A_{21}^T(1- \epsilon \sigma) \chi_{2,f}    \\& +\epsilon Q_{11}(1-\epsilon\sigma)  z_{1,f}  + Q_{12}(1-\epsilon\sigma) z_{2,f}   \\
\frac{dz_{2,f}}{d\sigma} =& - \epsilon A_{21}(1-\epsilon \sigma)  z_{1,f}- A_{22}(1- \epsilon\sigma) z_{2,f} \\&+ \epsilon S_{12}^T(1 - \epsilon\sigma)\chi_{1,f} + S_{22}(1-\epsilon\sigma)  \chi_{2,f}
\end{split}
\end{align}
\begin{align*}
\begin{split}
\frac{d\chi_{2,f}}{d\sigma}  =& \epsilon A_{12}^T(1- \epsilon\sigma ) \chi_{1,f} + A_{22}^T(1-\epsilon\sigma) \chi_{2,f}\\&  + \epsilon Q_{21}(1-\epsilon\sigma) z_{1,f}  + Q_{22}(1-\epsilon \sigma) z_{2,f}   ,
\end{split}
\end{align*}
and have an asymptotic expansion in $\epsilon$.  Furthermore, the following boundary conditions must be satisfied
\begin{align} \begin{split} \label{bdrycond}
z_{1,o}(0)+\epsilon z_{1,i}(0) =& z_{1,0},\\
z_{2,o}(0)+z_{2,i}(0) = &z_{2,0},\\
 \chi_{1,o}(1) + \epsilon \chi_{1,f}(0) = &\pi_{11}(z_{1,o}(1)+\epsilon z_{1,f}(0) )  \\&+\epsilon  \pi_{12}( z_{2,o}(1)+ z_{2,f}(0) 
\end{split}
\end{align}
\begin{align*}
\begin{split}
 \chi_{2,o}(1) +\chi_{2,f}(0) =& \pi_{21}(z_{1,o}(1)+\epsilon z_{1,f}(0) ) \\&+\pi_{22} (z_{2,o}(1)+ z_{2,f}(0) ) \end{split}
\end{align*}
along with the following limiting conditions
\begin{align*}
\lim_{\tau \rightarrow \infty} z^k_{1,i},z^k_{2,i}, \chi^k_{1,i},\chi^k_{2,i} &= 0,\\
\lim_{\sigma \rightarrow \infty} z^k_{1,f},z^k_{2,f}, \chi^k_{1,f},\chi^k_{2,f}  &= 0.
\end{align*}
for all $k=0,1,\dots$.  
\end{theorem}

The proof and construction are contained in \cite{OMalley}-\cite{OMalleyKung}; however, we give an outline of the construction in this paper for completeness.  

By matching the various orders of $\epsilon$ in the differential equations and boundary conditions that the outer, inner and final variables satisfy, one may determine the terms in the asymptotic expansion of the variables in  \eqref{expansionfull} up to any integer $N\geq 0$. 
It follows that the leading terms in the asymptotic expansion of the outer variables, $z_{1,o}^0$, $z_{2,o}^0$, $\chi_{1,o}^0$, and $\chi_{2,o}^0$ satisfy the system
\begin{align}\begin{split}\label{p opt condzero}
\frac{dz_{1,o}^0}{dt} & = A_{11}^0z_{1,o}^0 + A_{12}^0 z_{2,o}^0 -S_{11}^0\chi_{1,o}^0 -S_{12}^0\chi_{2,o}^0 , \\
\frac{d\chi_{1,o}^0}{dt} & = -(A_{11}^0 )^T\chi_{1,o}^0 - (A_{21}^0)^T\chi_{2,o}^0 - Q_{11}^0z_{1,o}^0 -Q_{12}^0 z_{2,o}^0, \\
0& = A_{21}^0 z_{1,o}^0 + A_{22}^0z_{2,o}^0 -(S_{12}^0 )^T\chi_{1,o}^0 -S_{22}^0\chi_{2,o}^0,  \\
0& = -(A_{12}^0)^T\chi_{1,o}^0 -(A_{22}^0)^T\chi_{2,o} - Q_{21}^0z_{1,o}^0 -Q_{22}^0z_{2,o}^0,\end{split}
\end{align}
with boundary conditions
\begin{align}\label{poptcond}
z_{1,o}^0 = z^0_{1,0}, \hspace{4mm} \chi_{1,o}^0 = \pi^0 z_{1,o}^0.
\end{align}
Higher order terms of the outer variables will satisfy non-homogeneous differential equations that are successively determined from the lower order terms.  Higher order boundary values are determined  from the lower order terms and the boundary  conditions in $\bold{P}$.  

The leading terms in the asymptotic expansion of the inner variables, $z_{1,i}^0$, $z_{2,i}^0$, $\chi_{1,i}^0$, and $\chi_{2,i}^0$ satisfy the system
\begin{align}
\begin{split} \label{ziqii}
\frac{dz_{1,i}^0}{d\tau} & = A_{12}^0(0) z_{2,i}^0- S_{12}^0(0) \chi_{2,i}^0\\
\frac{d\chi_{1,i}^0}{d\tau} & =  -  A_{21}^0(0)^T \chi_{2,i}^0 -Q_{12}(0)^0 z_{2,i}^0   \\
\frac{dz_{2,i}^0}{d\tau} & = A_{22}^0(0) z_{2,i}^0 - S_{22}^0(0) \chi_{2,i}^0\\
\frac{d\chi_{2,i}^0}{d\tau} & = -  A_{22}^0(0)^T \chi_{2,i}^0 -Q_{22}^0(0)^T z_{2,i}^0 .
\end{split}
\end{align}
 From assumptions (e), (g), and (h), we may obtain the general form of the decaying solution to $z_{2,i}^0$ and $\chi_{2,i}^0$
 \begin{align}\label{zchifast} \begin{split}
 z_{2,i}^0 (\tau) &= T_{11}(0)e^{-\Lambda(0)\tau}c,\\
 \chi_{2,i}^0(\tau) &=T_{21}(0)e^{-\Lambda(0)\tau}c, \end{split}
 \end{align}
 where $c$ is determined from the boundary condition for $z_{2,i}$ in \eqref{bdrycond} and is given by
 \begin{align}\label{cconstant}
 c=T_{11}^{-1}(0)(z_{2,0}^0 - z_{2,o}^0(0)).
 \end{align}
 Substituting  \eqref{zchifast} and \eqref{cconstant}  into the equations for $z_{1,i}^0$ and $\chi_{1,i}^0$ in \eqref{ziqii} and solving the resulting system yields the unique  solutions of $z_{1,i}^0$ and $\chi_{1,i}^0$.  Higher order terms for the inner variables can be obtained by matching powers of $\epsilon$ in \eqref{ziqi}.  The initial condition for $z_{2,i}^k(0)$ is determined from the outer term $z_{2,o}^{k-1}(0)$ for all $k=1,2,\dots$.  
 
 The leading terms in the asymptotic expansion of the final variables, $z_{1,f}^0$, $z_{2,f}^0$, $\chi_{1,f}^0$, and $\chi_{2,f}^0$ satisfy the system
\begin{align}
\begin{split} \label{ziqii2}
\frac{dz_{1,f}^0}{d\sigma} & = -A_{12}^0(1) z_{2,f}^0+ S_{12}^0(1) \chi_{2,f}^0\\
\frac{d\chi_{1,f}^0}{d\sigma} & =    A_{21}^0(1)^T \chi_{2,f}^0 +Q_{12}(1)^0 z_{2,f}^0   \\
\frac{dz_{2,f}^0}{d\sigma} & = -A_{22}^0(1) z_{2,f}^0 + S_{22}^0(1) \chi_{2,f}^0\\
\frac{d\chi_{2,f}^0}{d\sigma} & =  A_{22}^0(1)^T \chi_{2,f}^0 +Q_{22}^0(1)^T z_{2,f}^0 .
\end{split}
\end{align}
 From assumptions  (e), (g), and (h),  we may obtain the general form of the decaying solution to $z_{2,f}^0$ and $\chi_{2,f}^0$
 \begin{align}\label{zchifast2} \begin{split}
 z_{2,f}^0 (\sigma) &= T_{12}(1)e^{-\Lambda(1)\sigma}c_1,\\
 \chi_{2,f}^0(\sigma) &=T_{22}(1)e^{-\Lambda(1)\sigma}c_1, \end{split}
 \end{align}
 where $c_1$ is determined from the boundary condition for $\chi_{2,f}$  in \eqref{bdrycond} and is given by
 \begin{align}\label{cconstant2}
 c_1=(T_{11}(1)-\pi_{22}^0 T_{12}(1))^{-1}(\pi_{21}^0z_{1,o}^0+\pi_{22}^0z_{2,o}^0-\chi_{2,o}^0(1)).
 \end{align}
 Substituting  \eqref{zchifast2} and \eqref{cconstant2} into the equations for $z_{1,f}^0$ and $\chi_{1,f}^0$ in \eqref{ziqii2} and solving the resulting system yields the unique  solutions of  $z_{1,f}^0$ and $\chi_{1,f}^0$.  Higher order terms for the final variables can be obtained by matching powers of $\epsilon$ in \eqref{zfqf} and in the terminal condition for $\chi_{2,f}$ in \eqref{p bound cond1}.  

Note that the leading terms in the asymptotic expansion of the optimal control and states are dependent on the terms $z_{1,o}^0$ and $\chi_{1,o}^0$ which  satisfy a boundary value problem.  In order to obtain these terms, we use the  following theorem. 
\begin{theorem} \label{theoremreduced} The terms $z_{1,o}^0$ and $\chi_{1,o}^0$ satisfy the necessary optimality conditions of the following non-perturbed problem
\begin{align} \tag{\textbf{$\bar{\textbf{P}}$}}
\begin{cases}
 \underset{\hat{u}} \minimise &\bar{J}_\bold{P}(\hat{x}, \hat{u}),\\
 \mbox{subject to}& \frac{d\hat{x}}{dt}=\mathcal{A}\hat{x}+\mathcal{B}\hat{u},\\
 &\hat{x}(0)=z_{1,0}^0, 
\end{cases}
\end{align}
where the functional $ \bar{J}_\bold{P} (\hat{x}, \hat{u})$ is defined by 
\begin{equation*} \label{def jp reduced}
 \bar{J}_\bold{P} (\hat{x}, \hat{u})=\frac{1}{2} \int_0^1 \hat{x}^T\mathcal{Q}\hat{x}+\hat{u}^T\mathcal{R}\hat{u}\hspace{1mm} \mathrm{d}t  +\frac{1}{2} \hat{x}(1)^T\pi^0_{11}\hat{x}(1).
 \end{equation*}
Note that $\hat{x} \in W^{1,2}([0,1];\mathds{R}^m)$ as a function of $t$. We assume the matrix $\mathcal{Q}$ to be positive semi-definite and $\mathcal{R}$ to be positive definite where the matrices $\mathcal{Q}$,  $\mathcal{R}$,  $\mathcal{B}$ and $\mathcal{A}$ are defined as
\begin{align*} \begin{split} \label{qq}
\mathcal{R} =& R^0+(b_2^0)^T((A_{22}^0)^T)^{-1}Q_{22}^0(A_{22}^0)^{-1}b_2^0\\
\mathcal{Q}=&-(A^0_{21})^T((A^0_{22})^{-1})^TQ^0_{21} +Q^0_{11}-Q^0_{12}(A_{22}^0)^{-1}A_{21}^0\\
&+(A_{21}^0)^T((A_{22}^0)^{-1})^TQ_{22}^0(A_{22}^0)^{-1}A_{21}^0 -\mathcal{C} \mathcal{R}^{-1} \mathcal{C}^T \\
\mathcal{B}=&b_1^0-A^0_{12}(A^0_{22})^{-1}b_2^0\\
\mathcal{A}=& A^0_{11}-A^0_{12}(A_{22}^0)^{-1}A^0_{21}+\mathcal{B}\mathcal{R}^{-1}\mathcal{C}^T.\end{split}
\end{align*}
where 
\begin{equation}
\mathcal{C} = (Q_{12}^0-(A_{21}^0)^T((A_{22}^0)^T)^{-1}Q_{22}^0)(A_{22}^0)^{-1}b_2^0.
\end{equation}

\end{theorem}
The proof of Theorem \ref{theoremreduced}  follows from a straightforward application of the necessary conditions derived from the Hamiltonian function (see \cite{OMalleyKung} for details).    The assumption that $Q$ is positive semi-definite and $R$ is positive definite imply that the necessary optimality conditions  of $\bar{\textbf{P}}$ are sufficient; hence the variables $z_{1,o}^0$ and $\chi_{1,o}^0$ can be obtained from the optimal state and co-state of the problem $\bar{\textbf{P}}$ respectively.    

The asymptotic expansion of the optimal control $u$ for the primal problem $\bold{P}$ may now be obtained from the asymptotic expansion constructed for $\chi_{1}$ and $\chi_2$ and equation \eqref{controlu}.   The asymptotic expansion of the optimal control $\rho$ for the dual problem $\bold{D}$ is then given by  \eqref{bold lambda control} where $\hat{z}$ can be taken to be either the solution to the differential equations of $\bold{P}$ evaluated with  the asymptotic expansion of the optimal control to the primal problem or the approximation to $z$ given in \eqref{expansionfull}. 

\section{Proofs of main theorems}\label{main proofs}
In this section we provide the proofs of Theorems \ref{strong duality} and \ref{main theorem}.  The proof of Theorem \ref{main theorem} is split into two sections, the first of which covers the proof of an arbitrarily tight upper bound and the second of which covers the proof of an arbitrarily tight lower bound.

 In order to prove strong duality, we  must first derive the necessary optimality conditions for $\bold{D}$.  Omitting dependence on $t$ and $\epsilon$ for simplicity, the Hamiltonian function associated  with $\bold{D}$ is defined as
\begin{align*}
&H^\bold{D}(\hat{\gamma}, \hat{\rho},\hat{\mu})=\\
&-\frac{1}{2}( \hat{\rho}^TQ^{-1}\hat{\rho}+\hat{\gamma}^TbR^{-1}b^T\hat{\gamma})+\hat{\mu}^T(-A^T\hat{\gamma}+\hat{\rho}),\end{align*}  
where $ \hat{\mu} \in W^{1,2}([0,1]; \mathds{R}^{m+n} )$ as a function of $t$ is the co-state variable. The scaling $\hat{\mu} \rightarrow I^{\frac{1}{\epsilon}}\hat{\mu}$ recovers the standard form of the Hamiltonian. 
  Since there are no boundary conditions at the initial and final time, we have a normal multiplier.  
 Let $\rho$, $\gamma$, $\mu$ denote the optimal control, state and co-state respectively of $\bold{D}$. The necessary optimality conditions which these variables must satisfy are given by (see \cite{Clarke})
 \begin{align}\label{d opt cond}
\begin{split}
I^\epsilon\frac{d \mu}{dt} &=A \mu + bR^{-1}b^T \gamma,\\
I^\epsilon\frac{d \gamma}{dt}&=-A^T\gamma+\rho,
\end{split}
\end{align}
along with the boundary conditions
\begin{align}  \label{d bound cond1}
\mu(0,\epsilon)=z_0, \hspace{5mm} \mu(1,\epsilon) = -\pi(\epsilon)^{-1} I^\epsilon  \gamma(1,\epsilon).
\end{align}
As the optimal control in $\bold{D}$ is unconstrained, the Pontryagin maximum principle states that $\rho$ must satisfy the equality, $\frac{dH^{\bold{D}}}{d{\rho}}~=~0$. Hence
\begin{equation}\label{dual control}
\rho=Q\mu,
\end{equation}
where $\mu$ must satisfy the relevant optimality conditions given in \eqref{d opt cond} and \eqref{d bound cond1}.  

\begin{proof}[Theorem \ref{strong duality}]
Suppose that $z$, $\chi$, and $u$ denote the optimal state, co-state and control respectively of $\bold{P}$.   Consider the following definitions for $\gamma$, $\mu$, and $\rho$
\begin{align}  \label{equivalence}  \begin{split}
\gamma(t,\epsilon)&=-\chi(t,\epsilon), \\
 \mu(t,\epsilon)&=z(t,\epsilon), \\
 \rho(t,\epsilon)&=Q(t,\epsilon)z(t,\epsilon), \end{split}
\end{align} for $t\in [0,1]$ , $ \epsilon \in (0,\epsilon^*]$. We first show that $(\gamma, \rho)$ is a feasible solution for the dual problem and then show that this solution is optimal and that strong duality holds.

Substituting the definitions in \eqref{equivalence} into the differential equation for $\chi$ in  \eqref{p opt cond} and the boundary conditions  \eqref{p bound cond1} yields
\begin{align} \begin{split}
I^\epsilon \dot{\gamma}&=-A^T\gamma+\rho,\\
\label{mu z}
\mu(0,\epsilon)&=z_0,\\
\mu(1,\epsilon)&=-\pi(\epsilon)^{-1}I^\epsilon \gamma(1,\epsilon).
\end{split}
\end{align} The equations in  \eqref{mu z} are equivalent to the equations for $\gamma$ in \eqref{d opt cond} with boundary conditions in \eqref{d bound cond1} for $\bold{D}$.    Hence $(\gamma, \rho)$ is a feasible solution of the dual.   \\\\
  From weak duality, we know $  V_\bold{D}(\epsilon)\leq V_\bold{P}(\epsilon) $ (see \cite{EkelandTemam}, Chap. 2).  To show that there is a zero duality gap, we need to show $V_\bold{P}(\epsilon)=V_\bold{D}(\epsilon)$.  Evaluating \eqref{jd}  with the state and control given in \eqref{equivalence} along with the substitution   $\gamma(1)=-I^{\frac{1}{\epsilon}} \pi \mu(1) =-I^{\frac{1}{\epsilon}}\pi z(1)$  yields
  \begin{align*} \begin{split} 
&J_\bold{D}(Qz,-\chi,-\chi(0), -I^{\frac{1}{\epsilon}} \pi z(1), \epsilon)= \\
&\frac{1}{2}\int_0^1 \bigg{(} -z^TQz- \chi^T b R^{-1} b^T \chi \bigg{)}\hspace{1mm} \mathrm{d}t +\chi(0)^TI^\epsilon z_0\\
 &-\frac{1}{2} z(1)^T\pi z(1),\\
&=\frac{1}{2}\int_0^1 \bigg{(} -z^TQz -\chi^T b R^{-1} b^T \chi \hspace{-0.6mm}-\hspace{-0.6mm}\langle I^\epsilon\dot{\chi} ,z \rangle \hspace{-0.6mm}-\hspace{-0.6mm} \langle \chi ,I^\epsilon\dot{z}\rangle \hspace{-1mm}\bigg{)} \hspace{1mm} \mathrm{d}t \hspace{-0.6mm}\\&+\hspace{-0.6mm}\frac{1}{2} z(1)^T\pi z(1). \end{split}
\end{align*}
From the differential equations in \eqref{p opt cond}, we can evaluate the inner products in the integrand. Hence,
\begin{align*}
&J_\bold{D}(Qz, -I^{\frac{1}{\epsilon}} \pi z(1), \epsilon)= \frac{1}{2}\int_0^1 \bigg{(}z^TQz +\chi^Tb R^{-1} b^T\chi \bigg{)} \mathrm{d}t \\
&+\frac{1}{2} z(1)^T\pi z(1), \\
&= \frac{1}{2}\int_0^1\bigg{(}z^TQz + uRu \bigg{)} \hspace{1mm} \mathrm{d}t +\frac{1}{2} z(1)^T\pi z(1) \\
 &= V_\bold{P}(\epsilon).
\end{align*}
Since $(\gamma,\rho)$ is a feasible solution, by weak duality, we must have that $(\gamma,\rho)$  is optimal and $V_\bold{D}(\epsilon)=V_\bold{P}(\epsilon)$.

\end{proof}
 As the solution to  $\bold{P}$ is unique, we can justify Remark \ref{unique solution dual}, i.e. that a unique  solution exists for $\bold{D}$.
\subsection{Proof of Theorem \ref{main theorem}}
We split the proof of Theorem \ref{main theorem} into two subsections for parts (a) and (b) respectively. Part (c) follows immediately from parts (a) and (b) along with Theorem \ref{strong duality}.   
\subsubsection{Proof of Theorem \ref{main theorem} part (a)}
The proof of Theorem \ref{main theorem} part (a) follows from a simple integration of the differential equations in $\bold{P}$ with the approximate optimal control $u^N$ in \eqref{un1}.  From the definition given in \eqref{jp}, we obtain
  \begin{align}\label{vjp}\begin{split}
 &|V_\bold{P}(\epsilon)-J_\bold{P}(u^N,\hat{z}^N,\epsilon)|\\ 
&=\bigg{|}\frac{1}{2}\int_{0}^{1} \hspace{-1mm}z^TQz+u^TRu - (\hat{z}^N)^TQ\hat{z}^N\hspace{-1mm}- \hspace{-1mm}(u^N)^TRu^N\hspace{1mm}\mathrm{d}t\\&+\frac{1}{2}z(1,\epsilon)^T\pi(\epsilon)z(1,\epsilon) -\frac{1}{2}\hat{z}^N(1,\epsilon)^T \pi (\epsilon) \hat{z}^N(1,\epsilon) \bigg{|}, \end{split}
 \end{align}
 where $\hat{z}^N$ solves the differential equations in $\bold{P}$ with control given by $u^N$. Using the variation of parameters technique, we may write $\hat{z}^N$ and $z$ respectively as
\begin{align}\begin{split}\label{z1}
\hat{z}^N(t,\epsilon)&=\Phi_{I^{\frac{1}{\epsilon}} A}(t,0,\epsilon)z_0+\int_0^t \Phi_{I^{\frac{1}{\epsilon}} A}(t,s,\epsilon) I^{\frac{1}{\epsilon}} bu^N\mathrm{d}s,\\
z(t,\epsilon)&=\Phi_{I^{\frac{1}{\epsilon}} A}(t,0,\epsilon)z_0+\int_0^t \Phi_{I^{\frac{1}{\epsilon}} A}(t,s,\epsilon) I^{\frac{1}{\epsilon}}b u \hspace{1mm} \mathrm{d}s,
\end{split}
\end{align}
where $\Phi_{I^{\frac{1}{\epsilon}} A}$ is the resolvent matrix for the differential equations in \textbf{P}.  
Note that for a resolvent matrix $\Phi_x$ with $x\in \mathds{R}^{(m+n)\times(m+n)}$, the  following conditions must be satisfied for all $t\in [0,1]$  \vspace{2mm}
 \begin{enumerate}
 \item$ \frac{ d\Phi_x }{dt} = x(t,\epsilon) \Phi_x, \\$ 
 \item $\Phi_{x} (t,t,\epsilon) = I_{m+n},$\\
\item$ \det (\Phi_x) \neq 0.$
 \end{enumerate}\vspace{1mm}
 Let us partition the matrix $\Phi_{I^{\frac{1}{\epsilon}} A}$  as follows
\begin{align*}
\Phi_{I^{\frac{1}{\epsilon}} A} (t,s,\epsilon) = \begin{bmatrix} \phi_{11}  (t,s,\epsilon) &\phi_{12} (t,s,\epsilon) \\\phi_{21}  (t,s,\epsilon) &\phi_{22} (t,s,\epsilon)  \end{bmatrix},
\end{align*}
where $\phi_{11} \in \mathds{R}^{m\times m}$ and $\phi_{22}  \in \mathds{R}^{n\times n}$. 
  A well known result (see \cite{Flatto} and \cite{Smith}, Theorem 6.1.2) guarantees that the matrices $\phi_{ij}$ for $i,j=1,2$ satisfy the following bounds
  \begin{align}\begin{split}\label{phis}
  &||  \phi_{11}  (t,s,\epsilon)   ||_\infty \leq M_{11}, \hspace{0.5mm}   ||  \phi_{12}  (t,s,\epsilon)   ||_\infty \leq  \epsilon M_{12},\\
  &||  \phi_{21}  (t,s,\epsilon)   ||_\infty \leq M_{21}, \hspace{0.5mm}   ||  \phi_{22}  (t,s,\epsilon)   ||_\infty \leq  (\epsilon\hspace{-0.6mm} + \hspace{-0.6mm} e^{\frac{-k(t-s)}{\epsilon}}  )M_{22},
  \end{split}
  \end{align}
uniformly for all $0\leq s\leq t\leq1$, $\epsilon ~\in~ (0,\epsilon^*]$ where $k$ and $M_{ij}$ for $i,j=1,2$ are fixed, positive constants and the norm is defined as
\begin{equation*} \label{matrix norm}
||x||_\infty  = \max_{ij} |x_{ij}|,
\end{equation*}
for any matrix $x$.    As $b$ is  bounded for $t\in[0,1]$ and $\epsilon \in (0,\epsilon^*]$, it follows from  \eqref{un1}, \eqref{z1} and \eqref{phis} that
  \begin{equation} \label{z}
  z(t,\epsilon)=\hat{z}^N(t,\epsilon)+O(\epsilon^{N+1})\hspace{3mm} \mbox{as} \hspace{2mm}\epsilon \rightarrow 0,
  \end{equation}  uniformly on $[0,1]$. Substituting \eqref{un1} and  \eqref{z} into \eqref{vjp}, gives the inequality in
\eqref{pvjepsilon}. As $u$ is the optimal control, it follows that
\begin{equation*}
J_\bold{P}(u,z,\epsilon)\leq J_\bold{P}(u^N,\hat{z}^N,\epsilon).
\end{equation*}
Hence, we may conclude that  that the $O(\epsilon^{N+1})$ term in \eqref{pvjepsilon} satisfies $O(\epsilon^{N+1})<0$ for any integer $N\geq 0$ as $\epsilon \rightarrow 0$.

\subsubsection{Proof of Theorem \ref{main theorem} part (b)}
 From the definition in \eqref{jd}, we obtain
  \begin{align}\label{vdjd}\begin{split}
&\bigg{|}V_\bold{D} -J_\bold{D}({\rho}^N,\hat{\gamma}^N)\bigg{|} =\\
&\bigg{|}\frac{1}{2} \int_0^1 -\rho^TQ^{-1}\rho-\gamma^TbR^{-1}b^T\gamma  +({\rho}^N)^TQ^{-1}{\rho}^N+\\&(\hat{\gamma}^N)^TbR^{-1}b^T\hat{\gamma}^N  \mathrm{d}t -\gamma(0)^TI^\epsilon z_0 -\frac{1}{2} \gamma(1)^TI^\epsilon \pi^{-1} I^\epsilon \gamma(1)\\
 &+\hat{\gamma}^N(0)^TI^\epsilon z_0 +\frac{1}{2} \hat{\gamma}^N(1)^TI^\epsilon \pi^{-1} I^\epsilon \hat{\gamma}^N(1)  \bigg{|},
 \end{split}
 \end{align}
where $\hat{\gamma}^N$ is the state that satisfies the equations in $\boldsymbol{D}$ with control given by ${\rho}^N$ in \eqref{bold lambda control} and boundary condition \eqref{gamma1}.  We omit the dependence of \eqref{vdjd} on $t$ and $\epsilon$ for simplicity. Let  $\rho$ denote the optimal control of $\bold{D}$. It follows from \eqref{bold lambda control},  \eqref{equivalence}  and \eqref{z} that 
 \begin{equation} \label{control expansion}
\rho (t,\epsilon) ={\rho}^N(t,\epsilon)+O(\epsilon^{N+1}), \hspace{5mm} \mbox{ as } \epsilon \rightarrow 0,
 \end{equation}
 for any integer $N\geq 0$, uniformly on $[0,1]$. Furthermore, from \eqref{gamma1}, \eqref{d bound cond1}, \eqref{dual control}, and \eqref{control expansion}, it follows that 
  \begin{equation}\label{tildegamma}
 \gamma(1,\epsilon)=\hat{\gamma}^N(1,\epsilon)+O(\epsilon^{N+1}), \hspace{4mm} \mbox{ as } \epsilon \rightarrow 0,
 \end{equation}
 uniformly on $[0,1]$.  Let us define $\hat{\gamma}^N_\omega(\omega)= \hat{\gamma}(t)$ and $\gamma_\omega(\omega)=\gamma(t)$, where
\begin{equation}\label{ttrans}\omega=1-t.\end{equation}
Using the variation of parameters technique,  we may write $\hat{\gamma}^N_\omega$ and $\gamma_\omega$ respectively as

   \begin{align*}\label{gamma variation}\begin{split}
  \hat{ \gamma}^N_\omega(\omega)&=\hspace{-0.5mm}\Phi_{I^{\frac{1}{\epsilon}}A^T}(\omega,0,\epsilon) \hat{\gamma}^N_\omega(0)\hspace{-0.4mm}-\hspace{-1.4mm}\int_0^{\omega} \hspace{-1mm}\Phi_{I^{\frac{1}{\epsilon}}A^T}(\omega,r,\epsilon)I^{\frac{1}{\epsilon}}{\rho}^N\mathrm{d}r, \\
 \gamma_\omega(\omega)&=\hspace{-0.5mm}\Phi_{I^{\frac{1}{\epsilon}}A^T}(\omega,0,\epsilon) \gamma_\omega(0,\epsilon)\hspace{-0.4mm}-\hspace{-1.4mm}\int_0^{\omega} \hspace{-1mm}\Phi_{I^{\frac{1}{\epsilon}}A^T}(\omega,r,\epsilon)I^{\frac{1}{\epsilon}} \rho\hspace{1mm}\mathrm{d}r, 
 \end{split}
 \end{align*}
 where $\Phi_{I^{\frac{1}{\epsilon}}A^T}$ is the resolvent matrix for the differential equations in \eqref{d opt cond} under the transformation \eqref{ttrans}.   The matrix $\Phi_{I^{\frac{1}{\epsilon}}A^T}$  satisfies the bounds in \eqref{phis} for different fixed positive constants $k$ and $M_{ij}$ for $i,j=1,2$ (see \cite{Flatto} and \cite{Smith}, Theorem 6.1.2).   As $\Phi_{I^{\frac{1}{\epsilon}} A^T}(t,s,\epsilon)$ satisfies the bounds in \eqref{phis} for $s<t$, it follows from  \eqref{control expansion} - \eqref{tildegamma} that  
 \begin{equation}\label{eta t}
 \gamma(t,\epsilon)=\hat{\gamma}^N(t,\epsilon)+O(\epsilon^{N+1}), \hspace{4mm} \mbox{ as } \epsilon \rightarrow 0,
 \end{equation}
 uniformly on $[0,1]$.

Substituting  \eqref{control expansion} and \eqref{eta t}  into \eqref{vdjd}, we obtain the inequality in \eqref{o epsilon dual}. As $\rho$ is the optimal control, it follows that
\begin{equation*}
J_\bold{D}(\rho,\gamma,\epsilon)\geq J_\bold{D}({\rho}^N,\hat{\gamma}^N,\epsilon).
\end{equation*}
Hence, we may conclude that  that the $O(\epsilon^{N+1})$ term in \eqref{pvjepsilon} satisfies $O(\epsilon^{N+1})>0$ as $\epsilon \rightarrow 0$.

\section{Dual construction}\label{dual construction}
In this section, we outline the steps used to construct the dual problem $\bold{D}$. Following  the methods of  \cite{Alt}, \cite{Bertsekas} and \cite{Burachik}, we begin by converting the feasible set $\Sigma$ defined in \eqref{feasible set} into a subspace. We introduce dummy variables $\hat{s}_0,\hat{s}_1 \in \mathds{R}^{m+n}$ so that the problem $\bold{P}$ becomes
\begin{align*}
\begin{cases}
\underset{u\in U} \minimise &  J_\bold{P}(\hat{z},\hat{u},\epsilon , \hat{s}_0,\hat{s}_1),\\
\mbox{subject to} & (\hat{z},\hat{u},\hat{s}_0,\hat{s}_1)\in \boldsymbol{\Sigma},
\end{cases}
\end{align*}
where
\begin{align*}
J_\bold{P}(\hat{z},\hat{u},\epsilon ,\hat{s}_0,\hat{s}_1)=&\frac{1}{2}\int_{0}^{1} \hat{z}^TQ\hat{z}+\hat{u}^TR\hat{u}\hspace{0.5mm}\mathrm{d}t+\frac{1}{2}\hat{s}_1^T\pi(\epsilon)\hat{s}_1\\&+ \delta_{z_0}(\hat{s}_0) +\delta_U(\hat{u}),
\end{align*}
and 
\begin{align*}
 \boldsymbol{\Sigma}=\bigg{\{}&(\hat{z},\hat{u},\hat{s}_0,\hat{s}_1):I^\epsilon\frac{d\hat{z}}{dt}=A \hat{z}+ b\hat{ u}, \hat{z}(0,\epsilon)=\hat{s}_0,\\
& \hat{z}(1,\epsilon)= \hat{s}_1 , t\in [0,1], \epsilon \in (0,\epsilon^*]\bigg{\}}.
\end{align*} 
The $\delta$-function is defined by
\begin{equation*}
\delta_C(x) = \begin{cases} 0 & \mbox{if } x\in C,\\
+\infty & \mbox{otherwise} .\end{cases}  
\end{equation*} 
The feasible set $\boldsymbol{\Sigma}$ is now a closed subspace of $W^{1,2} \times W^{1,2} \times \mathds{R}^{n+m} \times \mathds{R}^{n+m}$.   Following standard techniques in duality theory, we introduce further dummy variables $\hat{v}_1~\in~ W^{1,2}([0,1];\mathds{R}^{m+n})$ , $\hat{v}_2\in W^{1,2}([0,1];\mathds{R}^{k})$, as functions of $t$, and  $\hat{\kappa}_0, \hat{\kappa}_1 \in \mathds{R}^{n+m}$   in order to dualise the problem.  The objective functional and conditions become
\begin{align} \label{newjp} \begin{split}
J_\bold{P}(\hat{v},\hat{w},\epsilon,\hat{\kappa}_0, \hat{\kappa}_1)=&\frac{1}{2} \int_0^1 \hat{v}_1^TQ\hat{v}_1+\hat{v}_2^TR\hat{v}_2\hspace{0.5mm}\mathrm{d}t +\frac{1}{2}\hat{\kappa}_1^T\pi(\epsilon)\hat{\kappa}_1 \\&+\delta_{z_0}(\hat{\kappa}_0) +\delta_U(\hat{w}),
\end{split}
\end{align}
subject to
\begin{align*}
&\hat{v}_1(t,\epsilon)=\hat{z}(t,\epsilon), \hspace{3mm} \hat{v}_2(t,\epsilon)=\hat{u}(t,\epsilon), \hspace{3mm}\hat{\kappa}_0=\hat{s}_0,\hspace{3mm} \hat{\kappa}_1=\hat{s}_1,\\& (\hat{z},\hat{u},\hat{s}_0,\hat{s}_1)\in \boldsymbol{\Sigma}.
\end{align*}
We separate the terms in \eqref{newjp} into the following four functions
\begin{align*}
f_1(\hat{v}_1)&=\frac{1}{2} \int_0^1 \hat{v}_1^TQ\hat{v}_1\mathrm{d}t,\\
f_2(\hat{v}_2)&=\frac{1}{2} \int_0^1 \hat{v}_2^TR\hat{v}_2\mathrm{d}t, \\
f_3(\hat{\kappa}_0)&=\delta_{z_0}(\hat{\kappa}_0),\\
f_4(\hat{\kappa}_1)&=\frac{1}{2}\hat{\kappa}_1^T\pi\hat{\kappa}_1.
\end{align*}
The dual functional $J_\bold{D}$ can be written by Fenchel duality \cite{Fenchel} as
\begin{align*}
&J_\bold{D}(\hat{\rho},\hat{\lambda}_2,\hat{\lambda}_3,\hat{\lambda}_4) =\\& \begin{cases}-f_1^*(\hat{\rho})\hspace{-0.5mm}-\hspace{-0.5mm}f_2^*(\hat{\lambda}_2)\hspace{-0.5mm} - \hspace{-0.5mm}f^*_3(\hat{\lambda}_3)\hspace{-0.5mm}- \hspace{-0.5mm}f^*_4(\hat{\lambda}_4)  & \hspace{-2mm} \mbox{if } (\hat{\rho},\hat{\lambda}_2,\hat{\lambda}_3,\hat{\lambda}_4) \in \boldsymbol{\Sigma}_1 ,\\ -\infty &\hspace{-2mm} \mbox{otherwise},\end{cases}
\end{align*}
where $\boldsymbol{\Sigma}_1$ is orthogonal to $\boldsymbol{\Sigma}$ and $f^*_1(\hat{\rho})$, $f_2^*(\hat{\lambda}_2)$, $f^*_3(\hat{\lambda}_3)$, and $f^*_4(\hat{\lambda}_4)$ are the Fenchel duals (see \cite{Bertsekas}, \cite{Burachik}) of $f_1(\hat{v}_1)$, $f_2(\hat{v}_2)$,$f_3(\hat{\kappa}_0)$ and $f_4(\hat{\kappa}_1)$  respectively. The Fenchel duals are defined as follows
\begin{align} \label{fenchel dual}
\begin{split}
f^*_1(\hat{\rho})&=\sup_{\hat{v}_1\in W^{1,2}} \int_0^1 \hat{\rho}^T\hat{v}_1- \frac{1}{2}\hat{v}_1^TQ\hat{v}_1 \mathrm{d}t,
 \\
f_2^*(\hat{\lambda}_2)&=\sup_{\hat{v}_2\in W^{1,2}} \int_0^1 \hat{\lambda}_2^T\hat{v_2}- \frac{1}{2} \hat{v_2}^TR\hat{v_2} \mathrm{d}t,\\
f^*_3(\hat{\lambda}_3)&=\sup_{\hat{\kappa}_0 \in \mathds{R}^{m+n}}  \hat{\lambda}_3(\epsilon)^T \hat{\kappa}_0 - \delta_{z_0}(\hat{\kappa}_0),\\
 f^*_4(\hat{\lambda}_4)&=\sup_{\hat{\kappa}_1 \in \mathds{R}^{m+n}}  \hat{\lambda}_4(\epsilon)^T \hat{\kappa}_1- \frac{1}{2}\hat{\kappa}_1^T\pi\hat{\kappa}_1.
\end{split}
\end{align}
Hence, the dual problem can be written as
\begin{align} \label{Dual Functional 2}
\hspace{2mm}\begin{cases}
\hspace{4mm}\underset{\hat{\rho},\hat{\lambda}_2,\hat{\lambda}_3,\hat{\lambda}_4} \maximise& -f^*_1(\hat{\rho})-f_2^*(\hat{\lambda}_2)-f_3^*(\hat{\lambda}_3)- f_4^*(\hat{\lambda}_4),\\ 
\hspace{4mm} \mbox{subject to } &(\hat{\rho},\hat{\lambda}_2,\hat{\lambda}_3,\hat{\lambda}_4)\in \boldsymbol{\Sigma}_1,
 \end{cases}
\end{align}
We will simplify the problem in \eqref{Dual Functional 2} by finding the solutions to the functions in \eqref{fenchel dual}. 
We may rewrite $-f^*_1(\hat{\rho})$  and $-f^*_2(\hat{\lambda}_2)$ respectively as
\begin{align*}
-f^*_1(\hat{\rho})= \inf_{\hat{v_1}(t,\epsilon)\in W^{1,2}}\int_0^1 F_1(\hat{v_1},t,\epsilon) dt,\\
-f^*_2(\hat{\lambda}_2)= \inf_{\hat{v_2}(t,\epsilon)\in W^{1,2}}\int_0^1 F_2(\hat{v_2},t,\epsilon) dt,
\end{align*}
where 
\begin{equation}\begin{split}\label{f*1}
F_1(\hat{v_1},t,\epsilon)&=  \frac{1}{2} \hat{v_1}^TQ\hat{v_1} - \hat{\rho}^T\hat{v_1},\\
F_2(\hat{v_2},t,\epsilon)&=  \frac{1}{2} \hat{v_2}^TR\hat{v_2} - \hat{\lambda}_2^T\hat{v_2}.
\end{split}
\end{equation}   Let $v_i$, $i=1,2$ denote the optimal solution to the corresponding equations in \eqref{f*1}.  From calculus of variations, we know that $v_i$, $i=1,2$ solves \eqref{f*1} if and only if it satisfies the respective Euler-Lagrange equation
\begin{equation*}
\frac{d}{dt}\frac{\partial F_i}{\partial \dot{v_i}}=\frac{\partial F_i}{\partial v_i}, \hspace{5mm} i=1,2.
\end{equation*}
Since $\frac{\partial F_i}{\partial \dot{v}_i}=0$, i=1,2, we have $\frac{\partial F_i}{\partial v_i}=0$ for $i=1,2$.  Hence
\begin{equation*}
\rho=Qv_1, \hspace{5mm} \lambda_2=Rv_2,
\end{equation*} It follows that 
\begin{align}\label{f1}\begin{split}
f^*_1(\hat{\rho})&=\frac{1}{2} \int_0^1 \hat{\rho}^TQ^{-1} \hat{\rho}\hspace{0.5mm}\mathrm{d}t,\\
f^*_2(\hat{\lambda}_1)&=\frac{1}{2} \int_0^1 \hat{\lambda}_2^TR(t,\epsilon)^{-1} \hat{\lambda}_2(t,\epsilon)dt.
\end{split}
\end{align}
Solving $ f_3^*(\hat{\lambda}_3)$ and $f_4^*(\hat{\lambda}_4)$ we obtain,
\begin{align}\begin{split}
\label{f*3}f_3^*(\hat{\lambda}_3)&=  \hat{\lambda}_3(\epsilon)^Tz_0,\\ 
 f_4^*(\hat{\lambda}_4)&= \frac{1}{2}\hat{\lambda}_4(\epsilon)^T\pi(\epsilon)^{-1} \hat{\lambda}_4(\epsilon).\end{split}
\end{align}
Substituting \eqref{f1} and  \eqref{f*3} into \eqref{Dual Functional 2}, the dual problem becomes
\begin{align}
\begin{cases} \label{almostjd1}
 \maximise\limits_{(\hat{\rho},\hat{\lambda}_2,\hat{\lambda}_3,\hat{\lambda}_4)} &\hspace{-3mm}J_\bold{D}(\hat{\rho},\hat{\lambda}_2,\hat{\lambda}_3,\hat{\lambda}_4,\epsilon),\vspace{3mm} \\ 
\hspace{2mm} \mbox{subject to } &\hspace{-3mm}(\hat{\rho},\hat{\lambda}_2,\hat{\lambda}_3,\hat{\lambda}_4)\in \boldsymbol{\Sigma}_1,
\end{cases}
\end{align}
where
\begin{align} \label{almostjd}
J_\bold{D}\hspace{-0.5mm}=\hspace{-1mm}\frac{1}{2} &\displaystyle\int_0^1 - \hat{\rho}^TQ^{-1}\hat{\rho}- \hat{\lambda}_2^TR^{-1} \hat{\lambda}_2 \mathrm{d}t- \hat{\lambda}_3(\epsilon)^T z_0 \\
&-\frac{1}{2}  \hat{\lambda}_4(\epsilon)^T\pi(\epsilon)^{-1}  \hat{\lambda}_4(\epsilon).
\end{align}
The following lemma will allow us to prove Theorem \ref{strong duality}.  The derivation below is similar to that in \cite{Burachik} but modified to allow for the singularly perturbed dynamics.
\begin{lemma} \label{lemmadual}
The subspace $\boldsymbol{\Sigma}_1$, orthogonal to $\boldsymbol{\Sigma}$, is given by
\begin{align*}\label{Dual Subspace}
\begin{split}
\boldsymbol{\Sigma}_1=\bigg{\{}& (\hat{\rho},\hat{\lambda}_2,\hat{\lambda}_3,\hat{\lambda}_4):  \hat{\rho}=I^\epsilon \frac{d\hat{\gamma}}{dt}+A^T\hat{\gamma},\hspace{2mm}\hat{\lambda}_2=b^T\hat{\gamma}, \\& \hat{\lambda}_3= I^\epsilon \hat{\gamma } (0), \hspace{2mm} \hat{\lambda}_4=-I^\epsilon  \hat{\gamma}(1),\hspace{2mm} t\in[0,1], \hspace{2mm}\epsilon \in(0,\epsilon^*]\bigg{\}},
\end{split}
\end{align*}
where
\begin{equation*}
\hat{\gamma}(t,\epsilon)=-  I^{\frac{1}{\epsilon}}\bigg{(} \int_t^1 \Phi_{I^{\frac{1}{\epsilon}}A}(s,t)^T\hat{\rho}\hspace{0.5mm}\mathrm{d}t - \Phi_{I^{\frac{1}{\epsilon}}A}(1,t)^T I^{\epsilon} \hat{\gamma}(1) \bigg{)}.
\end{equation*}
and $\Phi_{I^{\frac{1}{\epsilon}}A}$ is the resolvent matrix of the differential equations in $\bold{P}$.  \end{lemma}
\begin{proof} 
For simplicity we omit the dependence on $\epsilon$ for all variables and we omit the $\hat{\cdot}$ notation. Let $(\rho,\lambda_2,\lambda_3,\lambda_4) \hspace{-1mm}~\in~\hspace{-1mm} \boldsymbol{\Sigma}_1$.  As $\boldsymbol{\Sigma}_1$ is orthogonal to $\boldsymbol{\Sigma}$
\begin{equation} \label{orthogonal}
\lambda_4^Ts_1+\lambda_3^Ts_0+\int_0^1 \rho^Tz+\lambda_2^Tu \hspace{0.5mm} \mathrm{d}t=0 .
\end{equation}
The solution to the differential equations in \textbf{P} for time $t$ and time $t=1$ is given by
\begin{align}\label{s0}
\begin{split}
z(t)&=\Phi_{I^{\frac{1}{\epsilon}}A}(t,0)s_0 +\int_0^t \Phi_{I^{\frac{1}{\epsilon}}A}(t,s) I^{\frac{1}{\epsilon}}bu\hspace{0.5mm} \mathrm{d}s,\\
s_1=z(1)&=\Phi_{I^{\frac{1}{\epsilon}}A}(1,0)s_0 +\int_0^1 \Phi_{I^{\frac{1}{\epsilon}}A}(1,s) I^{\frac{1}{\epsilon}}bu \hspace{0.5mm} \mathrm{d}s,
\end{split}
\end{align}
respectively. Note that the equations in \eqref{s0} are integrable because the resolvent matrix satisfies the bounds in \eqref{phis}. Substituting  \eqref{s0} into \eqref{orthogonal}  and changing the order of integration yields
\begin{align*}
&0=\bigg{(}\lambda_4^T\Phi_{I^{\frac{1}{\epsilon}}A}(1,0) +  \lambda_3^T+ \int_0^1 \rho^T\Phi_{I^{\frac{1}{\epsilon}}A}(s,0)ds\bigg{)} s_0+\\& \int_0^1 \hspace{-1mm} \bigg{(}\lambda_2^T \hspace{-1mm} + \hspace{-1mm} \bigg{(} \int_s^1  \rho^T \Phi_{I^{\frac{1}{\epsilon}}A}(t,s)\mathrm{d}t + \lambda_4^T\Phi_{I^{\frac{1}{\epsilon}}A} (1,s)   \hspace{-1mm}  \bigg{)} I^{\frac{1}{\epsilon}} b \bigg{)} u \mathrm{d}s.
\end{align*}
Since $s_0$ and $u$ are arbitrary we obtain the following equations
\begin{align}
 \label{orthogonal 1}&0= \lambda_4^T\Phi_{I^{\frac{1}{\epsilon}}A}(1,0)   +  \lambda_3^T+ \int_0^1 \rho(s)^T\Phi_{I^{\frac{1}{\epsilon}}A}(s,0)ds,\\
\label{orthogonal 2} &0=\lambda_2(s)\hspace{-0.8mm}+\hspace{-0.8mm}b(s)^T I^{\frac{1}{\epsilon}}\bigg{(} \int_s^1\hspace{-1.8mm}  \Phi_{I^{\frac{1}{\epsilon}}A}(t,s)^T \rho\mathrm{d}t    + \Phi_{I^{\frac{1}{\epsilon}}A}(1,s)^T \lambda_4\bigg{)}.
\end{align}
Define
\begin{equation} \label{gamma}
\gamma(s)=-I^{\frac{1}{\epsilon}}\bigg{(} \int_s^1  \Phi_{I^{\frac{1}{\epsilon}}A}(t,s)^T \rho \hspace{0.5mm}\mathrm{d}t + \Phi_{I^{\frac{1}{\epsilon}}A}(1,s)^T \lambda_4 \bigg{)} .
\end{equation}
Substituting \eqref{gamma} into \eqref{orthogonal 2}
\begin{equation} \label{lambda 2}
\lambda_2(s)= b^T(s) \gamma(s).
\end{equation}
Setting $s=1$ in \eqref{gamma} we get 
\begin{equation} \label{lambda 4}
\gamma(1)= - I^{\frac{1}{\epsilon}} \lambda_4.
\end{equation}   Note that our expression for $\gamma$  in \eqref{gamma} is the general solution to
\begin{equation} \label{lambda 11}
 I^\epsilon \dot{\gamma}=-A^T\gamma+\rho.
\end{equation}
Rearranging   \eqref{lambda 11} and substituting into \eqref{orthogonal 1}, 
\begin{align*}
\lambda_3^T=&-\int_0^1\dot{\gamma}^T I^\epsilon \Phi_{I^{\frac{1}{\epsilon}}A}(t,0)\mathrm{d}t  -\lambda_4^T     \Phi_{I^{\frac{1}{\epsilon}}A}(1,0)\\&-\int_0^1\gamma(t)^TA\Phi_{I^{\frac{1}{\epsilon}}A}(t,0) \mathrm{d}t.
\end{align*}
Integrating by parts gives us
\begin{align} \begin{split} \label{lambda3equation}
\lambda_3^T&=-\gamma(1)^TI^\epsilon \Phi_{I^{\frac{1}{\epsilon}}A}(1,0)+ \gamma(0)^TI^\epsilon  -\lambda_4^T     \Phi_{I^{\frac{1}{\epsilon}}A}(1,0). \end{split}
\end{align}
Substituting \eqref{lambda 4}  into  \eqref{lambda3equation}
\begin{equation} \label{lambda 3}
\lambda_3^T=\gamma(0)^TI^\epsilon .
\end{equation}
From  \eqref{lambda 2}, \eqref{lambda 4}, \eqref{lambda 11}, \eqref{lambda 3}, we obtain the equations in $\boldsymbol{\Sigma}_1$.
\end{proof}
Substituting the results of Lemma \ref{lemmadual} into \eqref{almostjd1} and  \eqref{almostjd}, we obtain the dual formulation found in $\bold{D}$.

\section{Numerical Results} \label{numerical}
In this section, we illustrate the applicability of our method with three numerical examples. The first example concerns the Digital-Fly-by-Wire (DFW) program developed by NASA \cite{Elliott} for an F-8 aircraft.  This program interprets the pilots' flight path input signals and picks the flight control  that will achieve this path while accounting for both the aircraft dynamics and the various sensors relating to aircraft performance.  The second example is an optimal control problem over the clustered consensus network of $20$ nodes featured in \cite{Boker}, \cite{Boker2}.  In the first two examples, the solver was able to compute a solution to $\textbf{P}$; however, the results clearly illustrate the potential advantages of computing upper and lower bounds instead of the approximation $V^N$ satisfying \eqref{old old result}.   In the third example, we consider a larger network consisting of  $68$ nodes for which the solver failed to compute a solution.  However, in this case, we were still able to obtain upper and lower bounds.  The concept of  using clustered consensus networks to model physical systems has gained considerable attention in recent years and has been applied to a wide variety of areas such as social networks \cite{Wasserman}, wireless sensor networks \cite{Zhao}, vehicle formation \cite{Biyik}, power systems \cite{Romeres},  and electric smart grids \cite{Chakrabortty}.  

  All simulations for the first two examples and the reduced approximation of the third example were computed using GPOPS-II within Matlab.  The upper and lower bounds of the third example were computed using the ODE45 solver within Matlab. In each of the examples, we consider the $O(\epsilon)$ approximation and upper and lower bounds. 

\subsection{Example 1}
The singularly perturbed dynamics for this problem are taken from \cite{Nguyen}.
The states variables are defined as follows
\begin{align*}
&x_1 \hspace{2mm} - \hspace{2mm}  \mbox{velocity},\\
&x_2 \hspace{2mm} - \hspace{2mm} \mbox{angle of attack (rad)},\\
&y_1 \hspace{2mm} - \hspace{2mm} \mbox{pitch rate (rad/sec)},\\
&y_2 \hspace{2mm} - \hspace{2mm} \mbox{pitch angle (rad)},
\end{align*}
and the control variable is 
\begin{align*}
u \hspace{2mm} - \hspace{2mm} \mbox{stabilator deflection (rad)}.
\end{align*}
 The above variables represent the various quantities relative to a pre-defined equilibrium position of the aircraft.  
The matrices $I^\epsilon$, $Q$, $R$  $\pi$, $A_{ij}$ and $b_i$ for $i,j=12$ are defined as follows
\begin{align*}
I^\epsilon& = \begin{bmatrix}1 &0&0&0\\
0&1&0&0\\
0&0&\epsilon &0\\
0&0&0&\epsilon \end{bmatrix},\\
Q &= I_4 ,\hspace{1mm}, R=1, \hspace{1mm} \pi  =  I^\epsilon\\
A_{11}& = \begin{bmatrix}  -0.195378 & -0.676469\\ 1.478265 &0\end{bmatrix}, \\
A_{12} &= \begin{bmatrix} -0.917160 & 0.109033\\0&0 \end{bmatrix},\\
A_{21} &= \begin{bmatrix} -0.051601&0\\0.013579 &0 \end{bmatrix},\\
A_{22}& = \begin{bmatrix} -0.367954 & 0.43804\\-2.102596 & -0.21464\end{bmatrix}, \\
b_1 & = \begin{bmatrix} -0.023109 \\ -16.945030\end{bmatrix},\hspace{1mm}b_2= \begin{bmatrix} -0.048184\\ -3.810954 \end{bmatrix}.
\end{align*}
The initial conditions for the state variables are given by
\begin{equation*}
\begin{bmatrix} -2&3&-4&1\end{bmatrix}^T
\end{equation*}
Furthermore, we set $\epsilon = 0.0336$ and the time interval to be $[0,1]$.
We wish to find the optimal control that returns the aircraft to the equilibrium position. 

The initial condition of the reduced problem is given by
\begin{equation*}
\begin{bmatrix} -2&3\end{bmatrix}^T.
\end{equation*}
The matrices associated with the reduced problem are defined as follows 
\begin{align*}
\mathcal{Q} &=  \begin{bmatrix}1.009401 &0\\
   0& 1\end{bmatrix} ,\hspace{1mm} \mathcal{R} =5.513834,\\
\mathcal{A} &= \begin{bmatrix} -0.143614 & -0.676469\\
1.050984&  0\end{bmatrix}  , \hspace{1mm} \pi_{11}^0  = I_2,\\
\mathcal{B} &= \begin{bmatrix} 1.375594\\
 -16.945030 \end{bmatrix}.
\end{align*}

The matrix $T$ in \eqref{tmatrix} is given by
\begin{equation*}
\begin{bmatrix}
 -0.1184  &  0.5421&   -0.1824  & -0.1176\\
    0.9515  & -0.4073   &-0.4093   &-0.9369\\
   -0.1579   & 0.7281   & 0.8931    &0.2075\\
    0.2360  & -0.1003  &  0.0399&    0.2556 \end{bmatrix},
    \end{equation*}
    the eigenvalues of the matrix $\Lambda(t)$ in \eqref{tlambda} are given by 
    \begin{align*}
    3.5244, \hspace{1mm} \mbox{ and } \hspace{1mm} 0.6663,
    \end{align*}
    for all $t \in [0,5]$.   The constants $c$ and $c_1$ in \eqref{cconstant} and \eqref{cconstant2} respectively are given by 
    \begin{align*}
    c= \begin{bmatrix}  3.5607 \\
   -0.7475 \end{bmatrix}, \hspace{1mm} \mbox{ and } \hspace{1mm} c_1= \begin{bmatrix}0.0765\\
   -0.0544 \end{bmatrix}.
    \end{align*}
   
The solution to the primal problem and the upper and lower bounds are plotted as $\epsilon \rightarrow 0$ in Fig. \ref{fig1}.  Furthermore, we have plotted the asymptotic approximation described in \cite{OMalley}.  It is clear that from approximately $\epsilon=0.035$, the upper and lower bound provide a better approximation to the solution than the approximation in  \cite{OMalley} and \cite{OMalleyKung}.   Hence, the bounds provide additional information to the practitioner which can often be of considerable value when considering a choice of approximation. 
\begin{figure}[!h]
\centering
\includegraphics[width=2.5in]{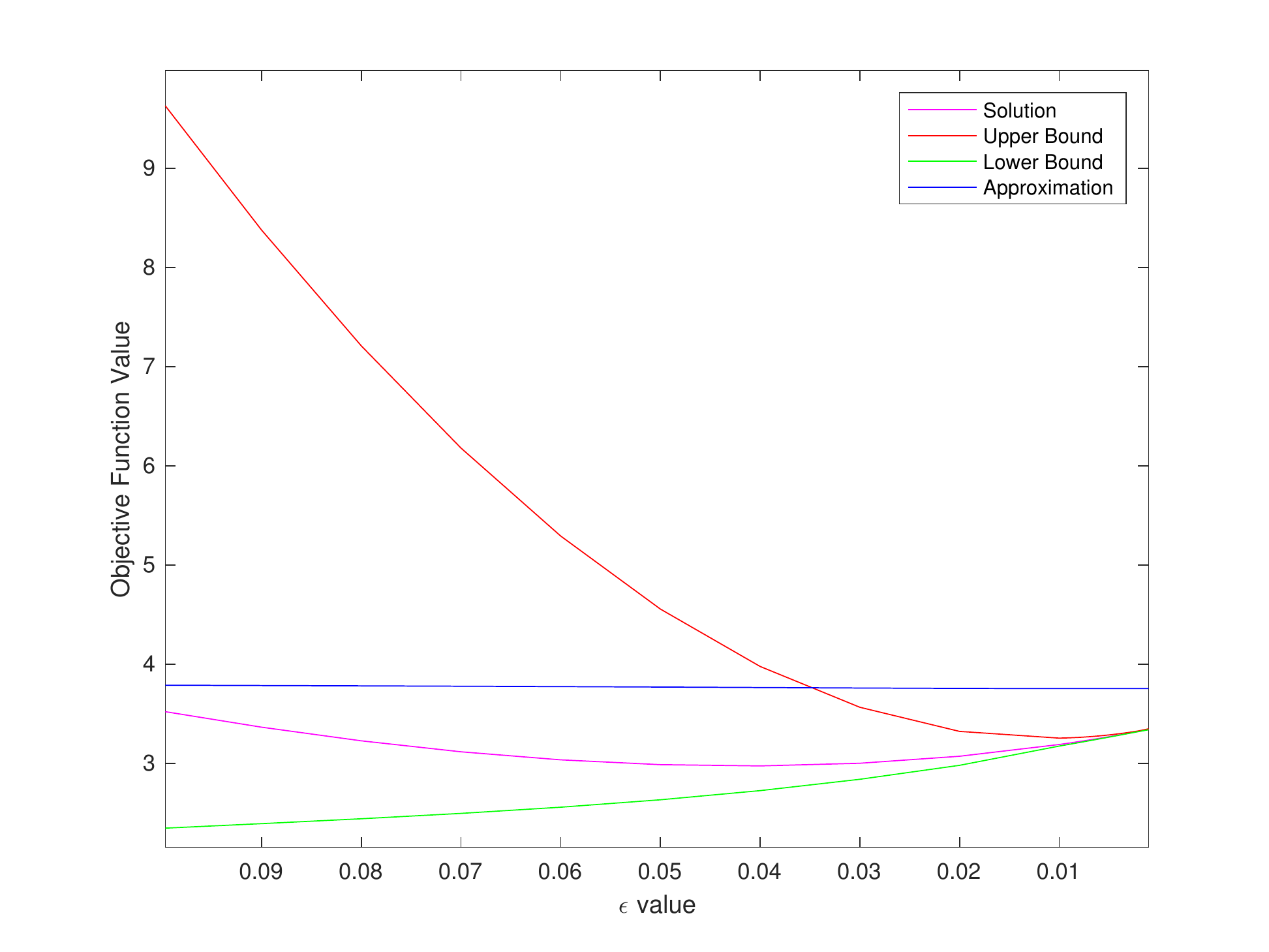}
\includegraphics[width=2.5in]{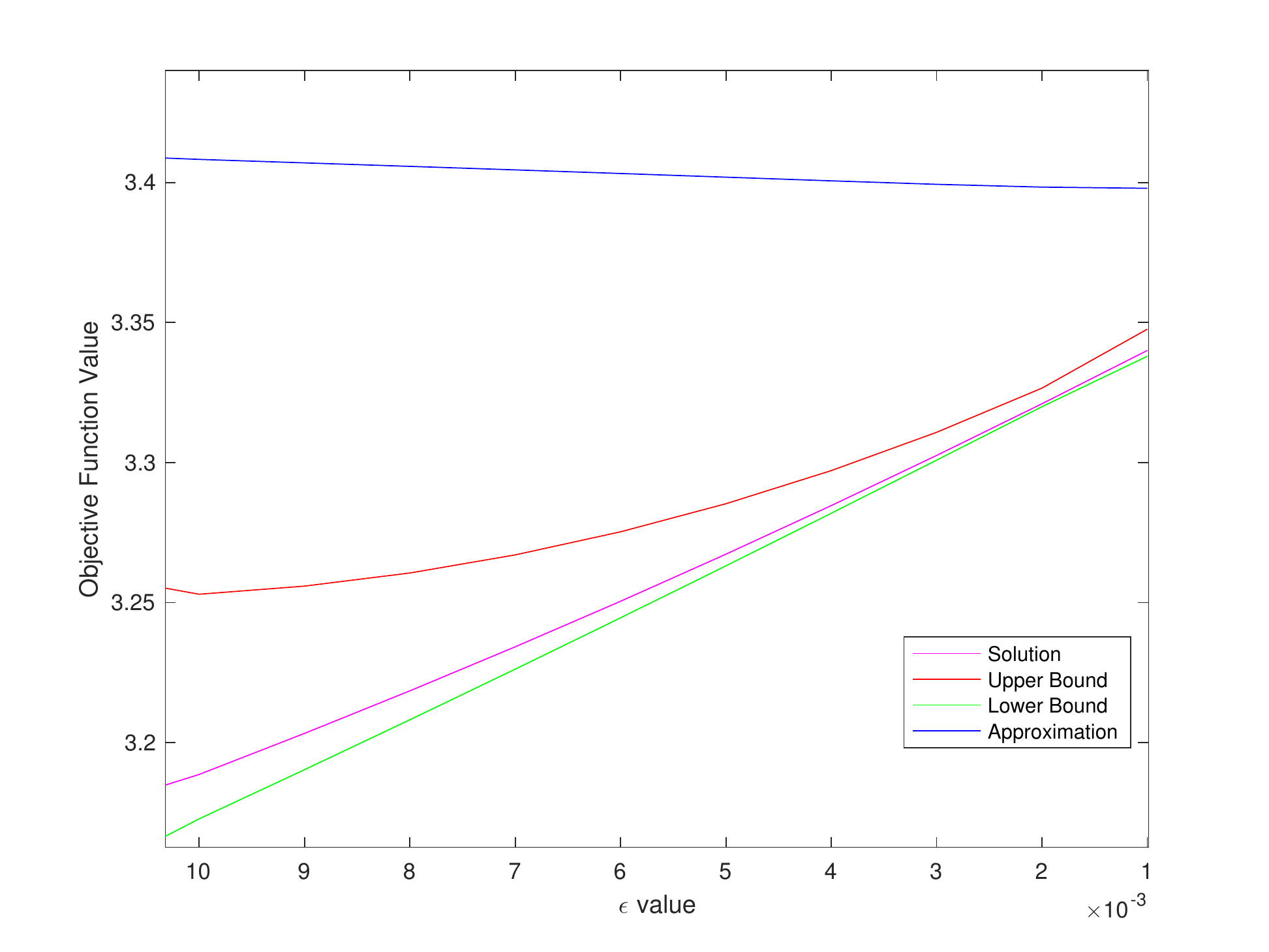}
 \caption{Upper and lower bound and approximation on the solution as $\epsilon \rightarrow 0$. }
 \label{fig1}
 \end{figure}
 
In Fig. \ref{fig2}, we have  plotted the difference between the upper and lower bounds. This graph shows that the upper and lower bounds are converging to the solution as $\epsilon \rightarrow 0$. 
   \begin{figure}[!t]
 \centering
\includegraphics[width=2.5in]{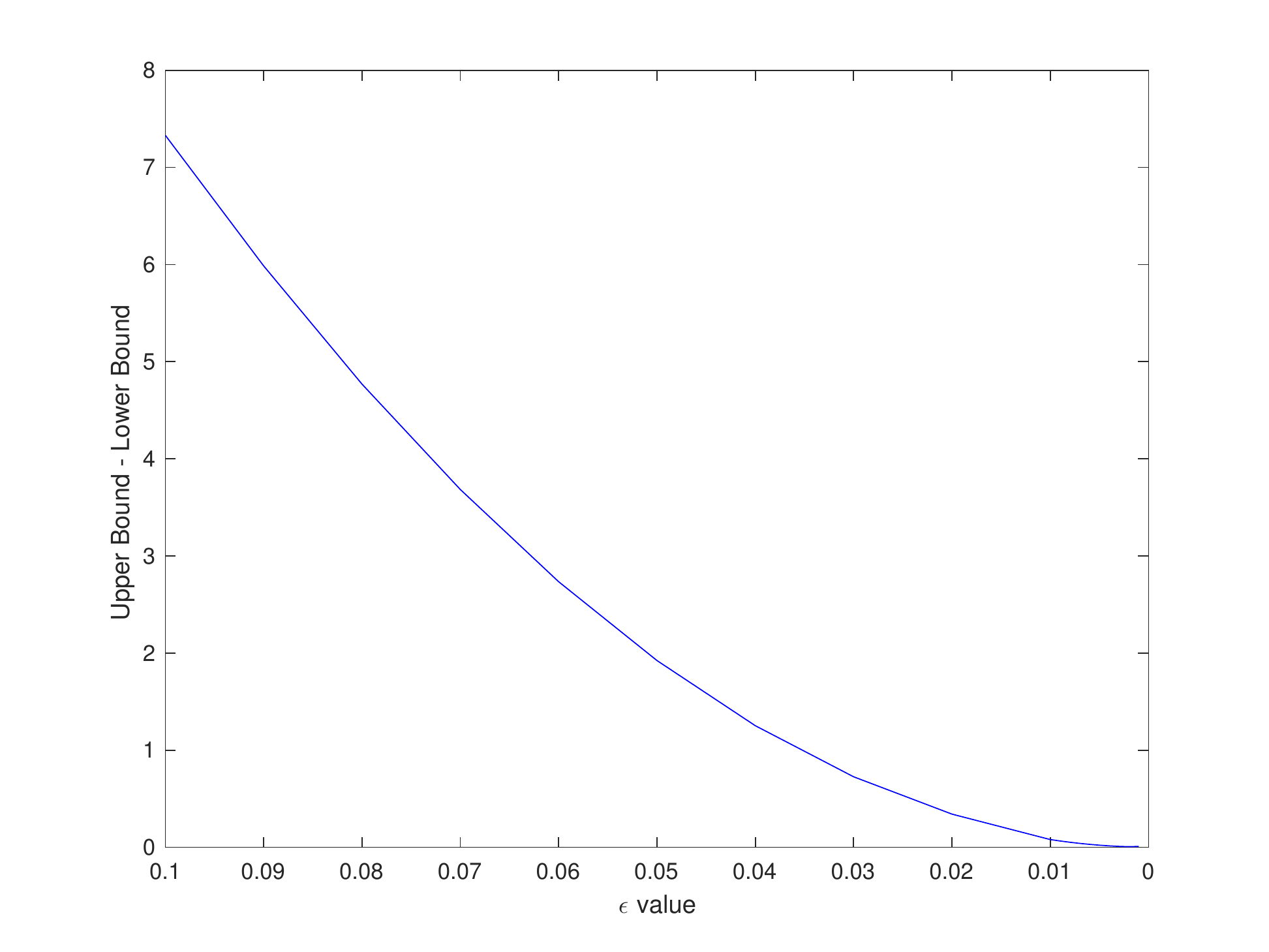}
 \caption{Difference between upper and lower bound as $\epsilon \rightarrow 0$.}
 \label{fig2}
  \end{figure}
  
\subsection{Example 2}
In this example, we consider an optimal control problem over the clustered consensus network in Fig. 3 with 20 nodes and 4 clusters.  This network was first considered in \cite{Boker} and \cite{Boker2}. The singular perturbation parameter is taken to be $\epsilon~ = ~0.25$. This parameter serves as the clustering parameter (see \cite{Boker}) associated with the network and represents the ratio of the number of internal connections to the number of external connections over all areas.
 \begin{figure}[h!]
\begin{center} \label{figd3}
\includegraphics[scale=0.4]{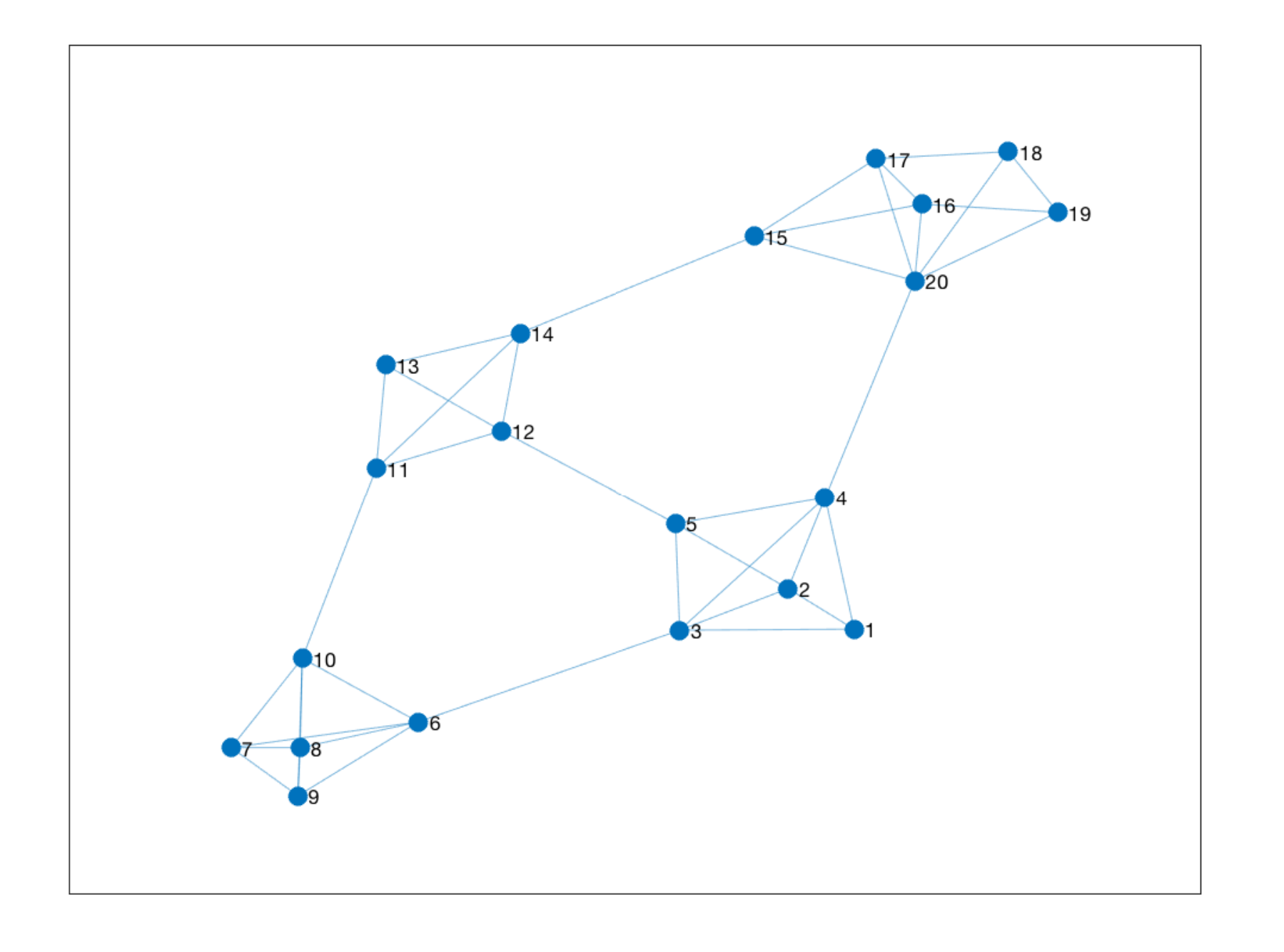}
\caption{20 node, 4 cluster network}
\end{center}
\end{figure}
The time horizon is 10 sec and the matrices in the objective functional $J_\bold{P}$ are given as $Q = I_{n}$, $R=I_{n}$, and $\pi = I^\delta \in\mathds{R}^{n \times n}$ where $n=20$. The matrices $A_{ij}$ and $b_i$ for $i,j=1,2$ are found using the method described in \cite{Boker} and \cite{Boker2} where the original control coefficient associated to each of the 20 nodes is 1.  The initial conditions for the nodes  were obtained randomly and are given by 
\begin{align} \label{initialcond}\begin{split} v(0)^T= [& 0.316, \hspace{0.5mm}
   0.959, \hspace{0.5mm}
    0.499, \hspace{0.5mm}
    0.739, \hspace{0.5mm}
    0.013, \hspace{0.5mm}
    0.605, \hspace{0.5mm}
   0.577, \\
   & 0.807,\hspace{0.5mm}
   0.655,  \hspace{0.5mm}
    0.878, \hspace{0.5mm}
    0.902, \hspace{0.5mm}
    0.152,\hspace{0.5mm}
        0.193,\hspace{0.5mm}
    0.791,\\
    &0.061, \hspace{0.5mm}
    0.39, \hspace{0.5mm}
    0.3,  \hspace{0.5mm}
    0.734, \hspace{0.5mm}
    0.104, \hspace{0.5mm}
    0.793]^T.
    \end{split}
    \end{align}
We assume the clustering parameter $\epsilon \rightarrow 0$ as the number of nodes in the network and the number of internal connections goes to infinity.   In this case, the matrices in the objective functional  of the reduced problem $\bar{J}_\bold{P}$ are given as $Q = I_{4}$, $R=I_{20}$, and $\pi = I^\delta \in\mathds{R}^{4 \times 4}$. The matrices in the dynamics are given by
\begin{align*}
 \mathcal{A} &= \begin{bmatrix}  -0.8000   & 0.2667 &   0.2667  &  0.2667\\
    0.2667  & -0.5333 &   0.2667      &   0\\
    0.3333 &   0.3333 &  -1 &   0.3333\\
    0.2222   &      0   & 0.2222 &  -0.4444 \end{bmatrix},\\
    \mathcal{B} & = \begin{bmatrix} 0.2667&0.2667&0.3333&0.2222 \end{bmatrix} V,
\end{align*}
where $V = blkdiag(\textbf{1}_{1\times 5}, \textbf{1}_{1\times 5}, \textbf{1}_{1\times 4}, \textbf{1}_{1\times 6})$ and $\textbf{1}_{i\times j}$ is the $i \times j$ matrix of ones.   The initial condition is given by
\begin{equation*}
[0.316, \hspace{0.5mm}
   0.959, \hspace{0.5mm}
    0.499, \hspace{0.5mm}
    0.739]^T.
\end{equation*}
We leave the computation of the matrix $T$ in \eqref{tmatrix}, the eigenvalues in \eqref{tlambda} and the constants in \eqref{cconstant} and \eqref{cconstant2} to the reader.  These values are easily obtained from the data available. 

 The solution to the original optimal control problem, the approximation using the method in \cite{OMalleyKung}, and our upper and lower bounds  are recorded in  Table \ref{Research Plan}.  
 \begin{table}[h]
\caption{Solution, reduced approximation and bounds for $\epsilon =0.25$.} \label{Research Plan} 
\begin{center}
\begin{tabular}{|c|c|c |c|} 
\hline
Solution & Reduced  &Upper Bd & Lower Bd\\ 
\hline
1.6925&1.4132 &1.7187 &1.5457\\
\hline\end{tabular}
\end{center}
\end{table}

Table \ref{Research Plan} illustrates the criterion for evaluating how good the reduced  approximation is.  It is clear that in this example, both the upper and lower bounds yield better approximations to the optimal solution than the approximation obtained in \cite{OMalleyKung}.  \subsection{Example 3} 
In our last example, we consider another consensus network with $68$ nodes and $4$ clusters with $\epsilon =0.0125$.  The reduced problem is obtained by assuming that both the number of nodes in the network and the number of internal connections goes to infinity.  The details of the optimal control problem are given in Example 2 with $n=68$ and a time horizon of 10 sec.   The first 20 initial conditions are in \eqref{initialcond} and the remaining values are chosen randomly.    The results are summarised in Table \ref{ResearchPlan2}. 
 \begin{table}[h]
\caption{Reduced approximation and bounds for $\epsilon =0.0125$.} \label{ResearchPlan2} 
\begin{center}
\begin{tabular}{|c|c |c|} 
\hline
 Reduced  &Upper Bd & Lower Bd\\ 
\hline
0.6546& 0.7386&0.7102 \\
\hline\end{tabular}
\end{center}
\end{table}

In this case, the solver was not able to compute the solution to the optimal control problem.  However,  we obtained the reduced solution and the upper and lower bounds.    It is clear that our result is applicable in the cases when the solution to an optimal control problem is infeasible and furthermore, provides more information to the practitioner when choosing an approximation to the solution. Although the dimension of the network in both Example 2 and 3, and corresponding optimal control problem, is small, obtaining upper and lower bounds on the solution, instead of  using the approximation given by the reduced problem can useful for implementation for any value of $\epsilon$.

\section{Conclusion}
We have developed a methodology to compute an arbitrarily tight upper bound $\chi_u^N$ and lower bound $\chi_l^N$ on the solution of a SPOC problem satisfying 
\begin{equation*}
|\chi_i^N(\epsilon) -\chi_l^N(\epsilon)| = O(\epsilon^{N+1}), \hspace{4mm} \mbox{as } \epsilon \rightarrow 0,
\end{equation*}
for any positive integer $N$.   Our methodology is based on the  construction of  both a dual SPOC problem with a strong duality property and a reduced dimension problem.  From the optimal control of the reduced problem, we construct an approximate control that is asymptotically equivalent in $\epsilon$ to the solution of the primal  SPOC problem.  To obtain an arbitrarily tight upper bound, we evaluate the differential equations of the primal problem with this approximate control. To obtain an arbitrarily tight lower bound, we  construct a control that is asymptotically equivalent to the optimal control of the dual problem using the reduced dimension problem and evaluate the differential equations in the dual problem with this constructed control.  It is clear from our results that obtaining the upper and lower bounds significantly improves the amount of information available to the practitioner as the  bounds hold for all $\epsilon$ and hence provide, for all $\epsilon$,  a criterion that determines the quality of any approximation to the solution.

\ifCLASSOPTIONcaptionsoff
  \newpage
\fi



%
\bibliographystyle{IEEEtran}

%
\vspace{-1cm}
\begin{IEEEbiography}[{\includegraphics[width=1in,height=1.25in,clip]{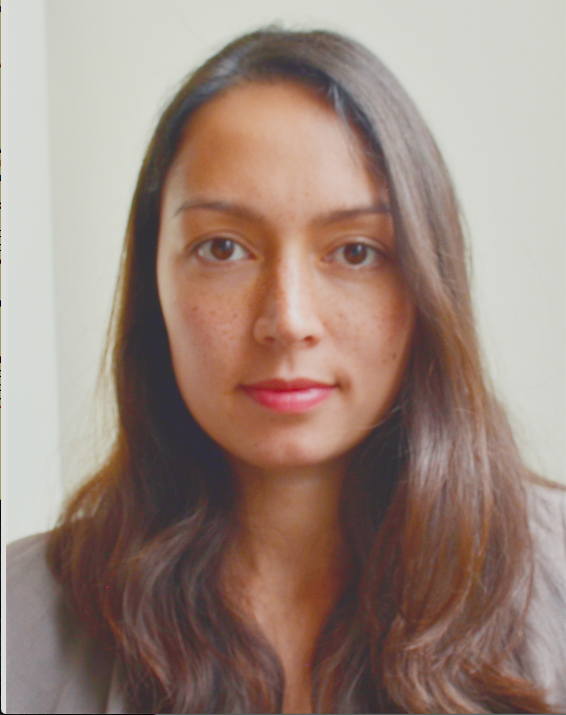}}]{Sei Howe}
received the B.A degree in mathematics from Reed College in 2011 and the M.Sc degree in pure mathematics from Imperial College London in 2012. She is currently a PhD candidate with the Department of Computer Science, Imperial College London.   

Her research interests include asymptotic analysis and singular perturbation theory with applications to optimal control and consensus networks.
\end{IEEEbiography}
\vspace{-0.5cm}
\begin{IEEEbiography}[{\includegraphics[width=1in,height=1.25in,clip]{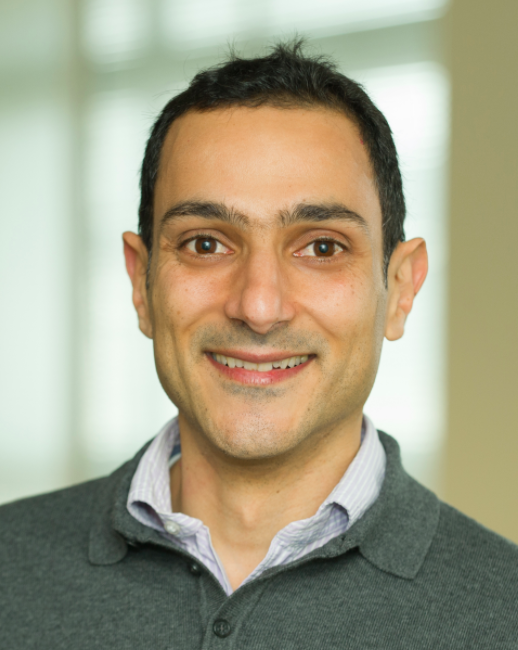}}]{Panos Parpas}
Panos Parpas is a Senior Lecturer in the Computational Optimization
Group of the Department of Computing at Imperial
College London. He is also a member of the Centre
for Process Systems Engineering at Imperial College London.
Before that, he was a research fellow at the Massachusetts Institute
of Technology (2009-2011). Dr. Parpas is interested in
the development of computational optimization methods.
He is particularly interested in multiresolution algorithms and decision making under uncertainty.
\end{IEEEbiography}
 \vfill 




\end{document}